\documentclass[a4paper,10pt]{article}
\usepackage{amsmath,amssymb}
\usepackage[utf8]{inputenc}
\usepackage[colorlinks,citecolor=blue]{hyperref}
\usepackage{color}
\usepackage{tikz}
\usepackage{scalefnt}
\usepackage{xspace}
\usepackage{framed}

\newtheorem{theorem}{Theorem}[section] 
\newtheorem{proposition}[theorem]{Proposition} 
\newtheorem{conjecture}[theorem]{Conjecture} 
\newtheorem{corollary}[theorem]{Corollary}
\newtheorem{lemma}[theorem]{Lemma}

\newtheorem{remark}[theorem]{Remark}

\newcommand{\QQ}{\mathbb{Q}}

\newcommand{\HH}{\mathsf{H}}
\newcommand{\fq}{\mathbb{F}_q}

\newcommand{\aff}[1]{A_{#1}}
\newcommand{\gm}{\mathbb{G}_m}
\newcommand{\WP}{\operatorname{WP}}

\newcommand{\gen}{\mathtt{generic}}
\newcommand{\ver}{\mathtt{versal}}
\renewcommand{\phi}{\varphi}
\newcommand{\aut}{\operatorname{Aut}}

\newcommand{\TA}{\mathbb{A}}
\newcommand{\TD}{\mathbb{D}}
\newcommand{\TE}{\mathbb{E}}

\newcommand{\red}{{red}\xspace}
\newcommand{\green}{{green}\xspace}
\newcommand{\orange}{{orange}\xspace}

\renewcommand{\mod}{\operatorname{mod}}

\newcommand{\ag}{\mathcal{G}}
\newcommand{\rk}{\operatorname{rk}}

\newenvironment{proof}{\begin{trivlist}\item{\bf{Proof.}}}
  {\hfill\rule{2mm}{2mm}\end{trivlist}}

\title{On some varieties\\ associated with trees}
\author{F. Chapoton}
\date{\today}

\begin{document}

\maketitle

\begin{abstract}
  This article considers some affine algebraic varieties attached to
  finite trees and closely related to cluster algebras. Their
  definition involves a canonical coloring of vertices of trees into
  three colors. These varieties are proved to be smooth and to admit
  sometimes free actions of algebraic tori. Some results are obtained
  on their number of points over finite fields and on their
  cohomology.
\end{abstract}

\tableofcontents

\section*{Introduction}

The theory of cluster algebras, introduced by S. Fomin and
A. Zelevinsky around 2000 \cite{cluster1, cluster2}, was motivated
initially by the study of total positivity in Lie groups and canonical
bases in quantum groups. It has since then developed rapidly in many
directions, among which one can cite (for example) triangulated
categories \cite{bmrrt}, triangulations of surfaces \cite{fst1} and Poisson
geometry \cite{gsv, gsv_book}.

Because cluster algebras are commutative algebras endowed with more
structure, it is natural to study them from the point of view of
algebraic geometry. The geometric study of cluster algebras has
nevertheless been mostly concentrated on aspects related to Poisson
geometry or symplectic geometry. The appearance of the known cluster
structure on coordinate rings of grassmannians in a physical context
\cite{arkani_hamed} has raised recently the interest in the
computation of integrals on the varieties associated with cluster
algebras. The natural context for this is of course the cohomology ring.

The present article aims to study some varieties closely related to
the spectrum of cluster algebras, and their cohomology rings. General
cluster algebras are defined using a quiver or a skew-symmetric
matrix. For our purposes, one needs as a starting point a presentation
by generators and relations of the cluster algebras. This is available
for cluster algebras with an acyclic quiver \cite{cluster3} and in a
few other cases (see for example \cite{mullerLA}). The choice has been
made here to restrict to a still smaller class, namely cluster
algebras with a quiver which is a tree, in the hope that the answers
may be simpler in that case, and also because all finite Dynkin
diagrams are trees.

Cluster algebras come with a subalgebra generated by so-called frozen
(or coefficient) variables, which are invertible elements. This
corresponds to a morphism from the spectrum of the cluster algebra to
an algebraic torus. At the start of this work, our intention was to
study both the fibers of this map and the spectrum in full. Later it
turned out that it is possible (for cluster algebras associated with
trees) to define more general varieties.

Cohomology and number of points on similar varieties have been
considered in some previous works \cite{gekhtman, mullerWP,
  nombre_points}. Some results of these articles will be recalled when
necessary.

\medskip

The article is organized as follows.

In the first section, one recalls a canonical tri-coloring of the
vertices of trees, originally defined in
\cite{coulomb_bauer,coulomb,zito} and not so well-known. This coloring
is closely connected to matchings and independent sets in the
trees. It will be used in an intensive way in the rest of
the article, as it enters in the very definition of the varieties
under study. One introduces the notion of red-green components of a
tree, and defines an important integer invariant, the dimension of a
tree.

The second section is devoted to the definition of the varieties. This
is rather involved, and the definition itself only appears after a
long preparation. One first considers a very general family of
varieties, depending one many invertible parameters. By considering
these varieties as objects in a groupoid, one can reduce this family
to a much smaller one, with less parameters. One proves that every
variety in the big family is isomorphic to a variety in the small
family. One also introduces an explicit condition of genericity. Then
everything is ready for the definition, which involves making an
independent choice for every red-green component of the tree.

The third section is devoted to some geometric properties of these
varieties. One proves by induction that all these varieties are
smooth, by finding explicit coverings by products of varieties of the
same type and algebraic tori. One next shows that some of these
varieties are endowed with a free torus action, which turns them into
principal torus bundles.

The fourth section turns to the study of the number of points over
finite fields. One shows by induction that the number of points is a
polynomial in the cardinality $q$ of the finite field. This is done by
finding an appropriate decomposition into pieces isomorphic to
products of varieties of the same type and algebraic tori. One then
gives formulas for some classical trees, including Dynkin
diagrams. One also obtains (Prop. \ref{decoupe_independante}) a
general decomposition as a disjoint union of products of tori and
affine spaces (indexed by independent sets), which allows to compute
the Euler characteristic.

The three next sections (5,6 and 7) deal with some computations
regarding the cohomology rings. Section 5 is a very short reminder
about known results about differential forms on varieties associated
with cluster algebras, and about the general theory of (mixed) Hodge
structure on the cohomology ring of algebraic varieties. Section 6
deals with some examples of trees, namely linear trees (the case of
which forms a useful building stone) and some trees of shape $H$ with
no parameters. Section 7 is about varieties where parameters have been
given a generic value. Our results about cohomology are rather
partial, restricted to special cases, but there does not seem to be
any simple general answer. The prominent missing case is in type $\TA$
with an odd number of vertices, where one proposes a conjecture.

The appendix A presents a simple algorithm for the computation of the
canonical coloring of trees. This algorithm is not needed in the rest
of the article.

\medskip

Let us finish this introduction by a few side remarks.

Another interesting question which has not been considered here is the study of the real points of the same varieties, and
their cohomology. This is probably also rather complicated, but
certainly worth looking at.

There seems to be some kind of vague analogy between the
counting-points polynomials considered here and the characteristic
polynomials of bipartite Coxeter elements (cf \cite{mcmullen} and
\cite{steko}), namely the general look and feel of these two families
of polynomials are similar in various points (including some relations
to Pisot and Salem numbers).

At the end of section 1.2 of \cite{spectra_book}, one can find some
speculations about the idea of ``quadratic spectra'' for graphs, that
would be an analog of the usual spectrum but related to quadratic
equations instead of linear equations. Maybe one can argue that the
cluster varieties considered here and their counting-point polynomials
are a good candidate for such a quadratic spectrum (even if they
involve polynomial relations of arbitrary degree).

\medskip

This work has been supported by the ANR program CARMA (ANR-12-BS01-0017-02).

\section{Combinatorics of trees}

\label{section1}

In this article, a \textbf{tree} is a finite connected and
simply-connected graph. A \textbf{leaf} is a vertex with at most one
neighbor. A \textbf{forest} is a disjoint union of trees.

\subsection{Canonical red-orange-green coloring of trees}

In this section, one recalls a canonical coloring of the vertices of
all trees, using the colors \red, \orange and \green. This coloring
has first appeared in an article by J. Zito \cite{zito} and has been
studied independently later by S. Coulomb and M. Bauer in
\cite{coulomb, coulomb_bauer}.

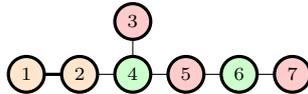
\begin{figure}[h]
\begin{center}
{\scalefont{0.7}
\begin{tikzpicture}[scale=0.7]
\tikzstyle{every node}=[draw,shape=circle,very thick,fill=white]

\draw (1,0) -- (2,0) -- (3,0) node[fill=red!20] {5} -- (4,0) node[fill=green!20] {6} -- (5,0) node[fill=red!20]{7};

\draw [very thick] (0,0) node[fill=orange!20] {1} -- (1,0) node[fill=orange!20] {2};

\draw (2,1) node[fill=red!20] {3} -- (2,0) node[fill=green!20] {4};
\end{tikzpicture}}
\end{center}
  \caption{Canonical coloring: \{1,2\} are orange, \{4,6\} green and \{3,5,7\} red}
  \label{fig:exemple3col}
\end{figure}

Let us consider a tree $T$. A \textbf{vertex cover} of $T$ is a subset
$S$ of the vertices of $T$ such that every edge of $T$ has at least
one end in $S$. A \textbf{minimum vertex cover} of $T$ is a vertex
cover of minimal cardinality among all vertex covers of $T$.

Let us use this notion to color the vertices of $T$ according to the
following rule: a vertex $v$ is
\begin{itemize}
\item \green if $v$ is present in all minimum vertex covers,
\item \orange if $v$ is present in some but not all minimum vertex covers,
\item \red if $v$ is present in no minimum vertex covers.
\end{itemize}
The colors have been chosen to match this definition with traffic lights colors.

For the tree of figure \ref{fig:exemple3col}, the minimum vertex
covers are made of the two \green vertices $\{4,6\}$ and one of the two \orange
vertices $\{1,2\}$.

\begin{remark}
  By taking the complementary subset, there is a bijection between
  minimum vertex covers, and sets of non-adjacent vertices of maximal
  cardinality (\textbf{maximum independent sets}, also called maximum
  stable sets).
\end{remark}

\smallskip

This coloring is also related to maximum matchings of $T$. A
\textbf{matching} of $T$ is a set $D$ of edges of $T$, such that every
vertex belongs to at most one element of $D$. The elements of $D$ will
be called \textbf{dominoes}. A \textbf{maximum matching} is a matching
of maximal cardinality among all matchings of $T$.

Then, a vertex $v$ is
\begin{itemize}
\item \green if $v$ is present in all maximum matchings, in several
  different dominoes.
\item \orange if $v$ is present in all maximum matchings, always in
  the same domino.
\item \red if $v$ is absent in some maximum matchings.
\end{itemize}

The proof of the equivalence of these two descriptions of the coloring
can be found in \cite{coulomb_bauer}.

For the tree of figure \ref{fig:exemple3col}, the maximum matchings are
made of three dominoes, one of them being the edge between the two
\orange vertices $\{1,2\}$.

\begin{proposition}
  \label{precise_matching}
  The \orange vertices are matched in pairs by the unique domino in
  which they are contained in any maximum matching. In maximum
  matchings, \green vertices are matched with \red vertices in several
  different ways.
\end{proposition}
\begin{proof}
  This is proved in \cite{coulomb_bauer}.
\end{proof}

\smallskip

This coloring has a third equivalent description, also given in
\cite{coulomb_bauer}.

It is the unique coloring of the vertices such that
\begin{itemize}
\item the induced forest on \orange vertices has a perfect matching,
\item every \green vertex has at least two \red neighbors,
\item every \red vertex has only \green neighbors.
\end{itemize}

It follows from this description that the coloring is stable by any of
the following operations:
\begin{itemize}
\item taking the induced forest on \orange vertices,
\item taking the induced forest on the union of \red and \green vertices,
\item removing a matched pair of \orange vertices,
\item removing a \green vertex.
\end{itemize}

An algorithm to compute the coloring is presented in appendix \ref{algosection}.
\begin{figure}\centering
  \includegraphics[height=4cm]{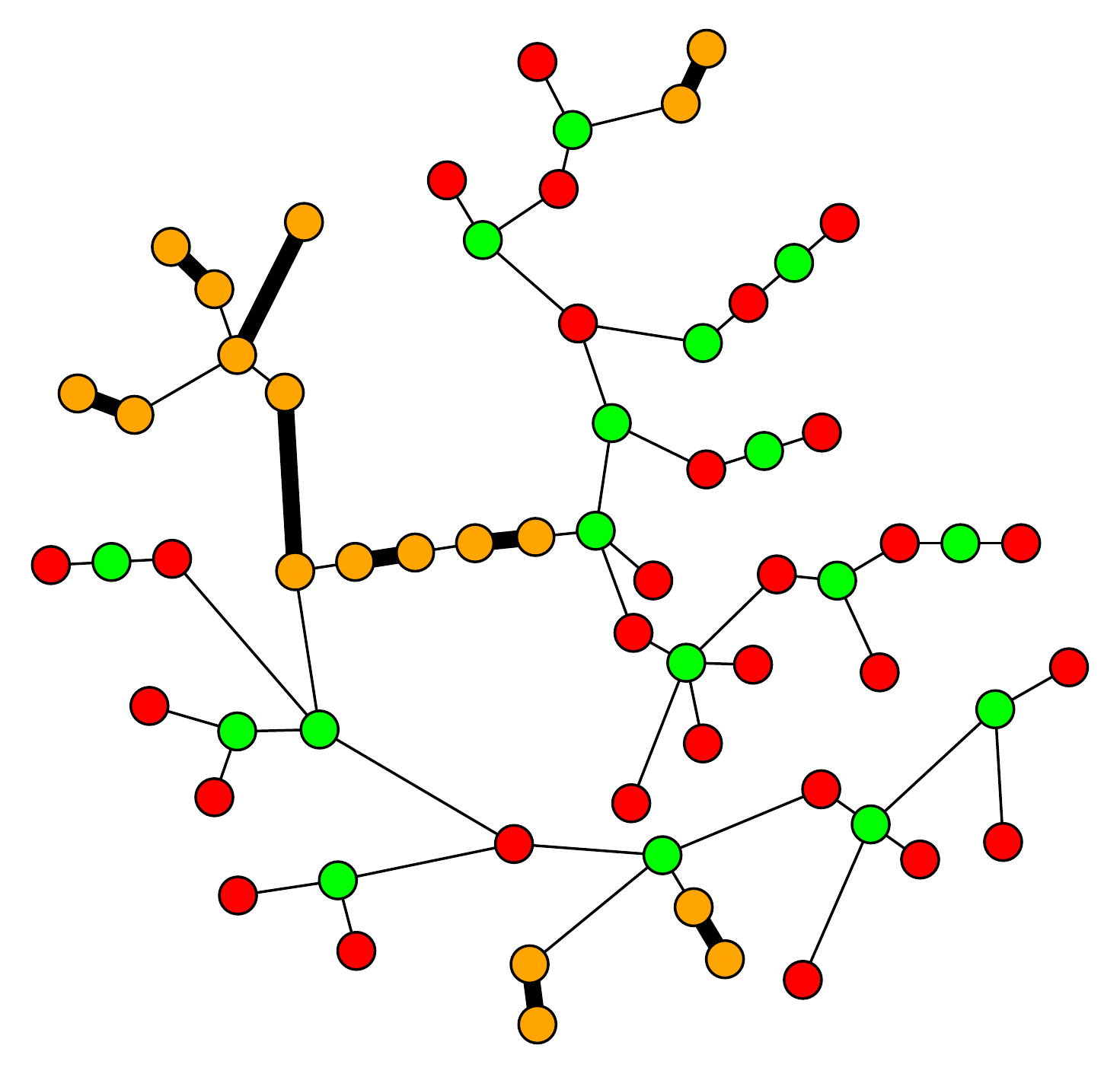}
  \caption{A typical example}
\end{figure}


\subsection{Further properties of the coloring}

Let us first state a corollary of the third description of the coloring.

\begin{lemma}
  A tree admits a perfect matching if and only if all vertices are \orange.
\end{lemma}
\begin{proof}
  If all vertices are \orange, there is a perfect matching by the
  first condition in the third description. If the tree has a perfect
  matching, letting all the vertices be \orange gives a coloring which
  satisfies all the required conditions, and therefore is the correct
  one by uniqueness.
\end{proof}

Note that the maximum matching is unique for these trees. We will call
them \textbf{orange trees}. They are also known as perfect trees or
matched trees \cite{simion}.

\smallskip

Let $T$ be a tree. The \textbf{red-green components} of $T$ are the
connected components of the graph defined by keeping only the edges of
$T$ with one \red end and one \green end. Every red-green component is
a tree, which is moreover bipartite with only red leaves. In these
trees, every green vertex has valency at least two.

This kind of trees has been considered under the name of $bc$-trees in
the study of blocks and cut-vertices of graphs, see for example
\cite[Chap. 4]{harary}.

Even trees with no \orange vertex can have several such components,
because there can be edges with two green ends, and these edges are not kept in the red-green components.

By the third description of the coloring, the coloring is stable by
taking a red-green component.

A tree which is equal to its only red-green component will be called a
\textbf{red-green} tree.

\begin{lemma}
  The set of maximum matchings of a tree is in bijection with the
  product of the sets of maximum matchings of its red-green
  components.
\end{lemma}
\begin{proof}
  The dominoes are fixed on the set of orange vertices, and cannot
  connect two distinct red-green components by proposition
  \ref{precise_matching}. Therefore, one can choose a maximum matching
  independently on every red-green component.
\end{proof}

Let us denote by $r(T), o(T)$ and $g(T)$ the number of \red, \orange
and \green vertices in the coloring of $T$. Let us call
\textbf{dimension} of a tree $T$ the quantity
\begin{equation}
  \label{defdim}
  \dim(T) = r(T) - g(T).
\end{equation}

\begin{remark}
  \label{rem_kernel}
  The dimension of $T$ is also the dimension of the kernel of the
  adjacency matrix of $T$, see \cite{coulomb_bauer}.
\end{remark}

\begin{lemma}
  \label{dim_is_not_covered}
  The dimension of $T$ is the number of vertices not covered by
  dominoes in any maximum matching.
\end{lemma}
\begin{proof}
  By the precise description of maximum matchings given in proposition
  \ref{precise_matching}, the number of dominoes in a maximum matching
  is $g(T)+o(T)/2$. The number of covered vertices is therefore $2
  g(T)+o(T)$. The statement follows.
\end{proof}

\begin{lemma}
  \label{about_dimension}
  The dimension of $T$ is the sum of the dimensions of the red-green
  components of $T$. Every red-green component has dimension at least
  $1$.
\end{lemma}
\begin{proof}
  The formula \eqref{defdim} for the dimension does not depend on
  \orange vertices, and is clearly additive on red-green components.

  Let $T$ be a red-green tree. The Euler characteristic is
  given by
  \begin{equation*}
    \chi(T) = 1 = r(T) + g(T) - e(T),
  \end{equation*}
  where $e(T)$ is the number of edges of $T$. On the other hand,
  \begin{equation*}
    e(T) \geq 2 g(T),
  \end{equation*}
  because every green vertex has at least two red neighbors.
\end{proof}

\begin{lemma}
  \label{remove_one_red_vertex}
  Let $T$ be a red-green tree. Let $F$ be the forest obtained by
  removing one red vertex of $T$. Then $\dim(F) = \dim(T) -1$.
\end{lemma}
\begin{proof}
  Removing the vertex makes a big difference in the colorings of $F$
  and $T$. The coloring of $F$ can be obtained from the restriction
  of the coloring of $T$ by some avalanche of orange vertices, as
  follows.
  
  At start, the restriction of the coloring of $T$ gives a bad
  coloring of $F$, where some green vertices $v$ may have exactly one
  red neighbor. If not, then the coloring is the canonical
  one. Otherwise, one can turn every such vertex $v$ and its unique
  red neighbor into an orange domino. Doing that may create a certain
  number of green vertices with exactly one red neighbor. For each of
  them, replace it and its unique red neighbor by a
  domino. Repeat this as long as there is some green vertex with
  exactly one red neighbor. This must stop at some point, because we
  work in a finite union of trees. At the end of this avalanche of
  orange dominoes, one has obtained a canonical coloring of $F$.

  This construction implies that the dimension of $F$ is the
  dimension of $T$ minus $1$, because it only involves turning pairs
  (green vertex, red vertex) into orange dominoes.
\end{proof}



\begin{lemma}
  \label{decoupe_dim}
  Let $T$ be a red-green tree and $u-v$ be any edge of $T$. Let $F$ be
  the forest induced from $T$ by removing the vertices $u$ and
  $v$. Then the dimension of $T$ is the sum of the dimensions of the
  trees in $F$.
\end{lemma}
\begin{proof}
  Assume that $u$ is green and $v$ is red. Let $S_1,\dots,S_k$ be the
  trees in $F$ attached to $u$ and let $T_1,\dots,T_\ell$ be the trees
  in $F$ attached to $v$.

  Then the coloring of every $S_i$ is just obtained by restriction,
  because it still satisfies the third description of the canonical
  coloring.

  On the other hand, let us denote by $\widehat{T}_j$ the tree obtained
  from $T_j$ by adding back the red vertex $v$. Then the coloring of
  every $\widehat{T}_j$ is just obtained by restriction, because it
  still satisfies the third description of the canonical coloring.

  By the definition \eqref{defdim} of the dimension, one therefore finds that
  \begin{equation*}
    \dim(T) = \sum_i \dim(S_i) + \sum_j (\dim(\widehat{T}_j)-1).
  \end{equation*}
  By lemma \ref{remove_one_red_vertex}, this is equal to the expected result.
\end{proof}

\smallskip

\begin{lemma}
  Let $T$ be a tree. Let $u-v$ be a red-green edge of $T$. There
  exists a maximum matching of $T$ containing $u-v$.
\end{lemma}
\begin{proof}
  One can assume that $T$ is a red-green component, as maximum
  matchings of different red-green components are independent.

  One can take maximum matchings of the connected components of the
  forest $F$ induced from $T$ by removing $u$ and $v$. From lemma
  \ref{decoupe_dim} and lemma \ref{dim_is_not_covered}, the number of
  vertices not covered on $F$ is the dimension of $T$. Therefore,
  adding the domino $u-v$ gives a maximum matching of $T$.
\end{proof}

\begin{lemma}
  \label{presque_pavage}
  Let $T$ be a red-green tree and let $v$ be a leaf of $T$. There
  exists a maximum matching of $T$ where the vertices which are not
  covered are leaves. Moreover, unless $T$ is reduced to the single
  vertex $v$, one can find such a matching where $v$ is in a domino.
\end{lemma}
\begin{proof}
  By induction on the size of the tree $T$. This is true for the tree
  with $1$ vertex.

  Let us call $u$ the \green neighbor of the red leaf $v$. The induced
  forest $F$ defined as $T\setminus \{u,v\}$ is made of red-green
  trees, whose sum of dimensions is the dimension of $T$ by lemma
  \ref{decoupe_dim}. By induction, one can find a maximum matching of
  $F$ such that vertices which are not covered are leaves of
  $F$. Moreover, one can choose this matching such that the vertices
  which are not covered are in fact leaves of $T$.

  One then obtains by adding the domino $u-v$ a
  maximum matching of $T$ with all the required properties.
\end{proof}

\begin{lemma}
  \label{rouge_dans_matching}
  Let $T$ be a tree and let $v$ be a red vertex of $T$. There
  exists a maximum matching of $T$ not containing $v$.
\end{lemma}
\begin{proof}
  Otherwise, one would get a contradiction with the characterization
  of the canonical coloring.
\end{proof}

\begin{lemma}
  \label{dimension1}
  The trees obtained by removing a leaf in an \orange tree are exactly
  the trees of dimension $1$. They have exactly one red-green
  component.
\end{lemma}
\begin{proof}
  Let us pick an orange tree $T$ and a leaf $v$ with adjacent vertex
  $w$. Removing the leaf $v$ gives a tree $T\setminus \{v\}$ with a matching
  covering all vertices but $w$. This is clearly a maximum matching,
  hence $T\setminus \{v\}$ has dimension $1$ by lemma \ref{dim_is_not_covered}.

  Conversely, consider a tree $T'$ of dimension $1$. It has exactly
  one red-green component, as every red-green component contributes at
  least $1$ to the dimension by lemma \ref{about_dimension}. This
  red-green component has dimension $1$. By lemma
  \ref{presque_pavage}, one can find a maximum matching of $T'$
  missing only one leaf $w$. Adding a vertex $v$ attached to $w$ gives
  a tree with a perfect matching, \textit{i.e.} an orange tree.
\end{proof}

We will call the trees of dimension $1$ \textbf{unimodal trees}.

\begin{remark}
  The classical Dynkin diagrams are simple examples of trees:
  \begin{itemize}
  \item Type $\TA_n$: \orange for even $n$, unimodal for odd $n$,
  \item Type $\TD_n$: unimodal for odd $n$,
  \item Type $\TE_n$: \orange for $n=6, 8$, unimodal for $n=7$.
  \end{itemize}
  The type $\TD_n$ with $n$ even has dimension $2$.
\end{remark}

\section{Affine algebraic varieties}

Using the coloring of the previous section, one can define several
affine algebraic varieties attached to a tree $T$ and some auxiliary
choices. These varieties are closely related to cluster algebras.

First, let us consider the system of equations
\begin{equation}
  \label{exchange_rel}
  x_i x'_i = 1 + \alpha_i \prod_{i-j} x_j
\end{equation}
for all vertices $i$ of $T$, where the product runs over vertices $j$
adjacent to $i$. Here $x_i$ and $x'_i$ are called \textbf{cluster
  variables}, and $\alpha_i$ are called \textbf{coefficient
  variables}.

By a special case of \cite[Corollary 1.17]{cluster3}, this system is a
presentation of the cluster algebra associated with the quiver given
by a bipartite orientation of $T$, with one frozen vertex attached to
every vertex of $T$ (in such a way that all vertices of $T$ remain
sources or sinks). In the context of cluster algebras, the equations
\eqref{exchange_rel} are called exchange relations.

We will be interested here in considering the $\alpha_i$ as
parameters, and letting them either vary in some well-chosen families
or take fixed generic values (and even a mix of these two
possibilities), so that the resulting space is smooth.

\subsection{Jumping around a groupoid}

Let us denote by $X_T(\alpha)$ the algebraic scheme defined by fixing
some invertible values for all coefficient variables $\alpha_i$.

Recall the following lemma (\cite[Lemma 2.2]{nombre_points}).
\begin{lemma}
  \label{jumpin}
  Let $u-v$ be an edge of $T$. Let $\beta$ be defined by
  $ \beta_w = \alpha_w/\alpha_u$ if $w$ is a neighbor of $v$ (in
  particular $ \beta_u = 1$) and $\beta_w = \alpha_w$ otherwise. Then
  $X_T(\alpha)$ and $X_T(\beta)$ are isomorphic, by the change of
  variables $x_v=\alpha_u x_v$ and $x'_v=x'_v/\alpha_u$.
\end{lemma}

One may say that the coefficient $\alpha_u$ has \textbf{jumped} away
from $u$ over $v$ and its inverse has got spread over all other
neighbors of $v$. When $v$ has $u$ as only neighbor, the coefficient
$\alpha_u$ just disappears from the equations.

From now on, we will only admit the following kinds of jumps:
\begin{itemize}
\item a red vertex over one of its green neighbors,
\item a green vertex over one of its red neighbors,
\item an orange vertex over its matched orange neighbor.
\end{itemize}

\smallskip

Let us now define a groupoid $G_T$ with objects the schemes
$X_T(\alpha)$ indexed by invertible values of the parameters
$\alpha$, and isomorphisms $X_T(\alpha) \simeq X_T(\bar{\alpha})$ of
the shape
\begin{equation}
  \begin{cases}
  \bar{x}_i &\mapsto \lambda_i x_i,\\
  \bar{x}'_i &\mapsto x'_i/\lambda_i,
  \end{cases}
\end{equation}
where $\lambda_i$ are some invertible elements. The parameters are then related by
\begin{equation}
  {\alpha}_i = \bar{\alpha}_i \prod_{i - j}\lambda_j .
\end{equation}

Note that every jump corresponds to an isomorphism in the groupoid $G_T$.

\smallskip

\begin{proposition}
  \label{reduction_via_matching}
  For every maximum matching $M$ and given parameters $\alpha$, there
  exists unique parameters $\beta$ (given by monic Laurent monomials
  in $\alpha$) such that
  \begin{itemize}
  \item the function $\beta$ is $1$ except on the set of red vertices
    not covered by $M$.
  \item $X_T(\alpha)$ is isomorphic to $X_T(\beta)$ by a sequence of jumps.
  \end{itemize}

  Moreover,
  \begin{itemize}
  \item[(a)] the function $\beta$ only depends on the values of $\alpha$ on
    the red vertices of $T$,
  \item[(b)] the values of $\beta$ on a red-green component are Laurent
    monomials in the values of $\alpha$ on the same red-green
    component.
  \end{itemize}
\end{proposition}
\begin{proof}
  Let us first prove the existence of such parameters $\beta$. The
  main idea is to iterate lemma \ref{jumpin} by jumping over dominoes of $M$.

  Let us define an auxiliary oriented graph $\ag$ as follows: the
  vertices of $\ag$ are the vertices of $T$, and there is an edge $u \to
  w$ in $\ag$ if $u-v$ is a domino in $M$ and $v-w$ is another edge in $T$.

  With this notation, if there are edges starting from $u$ in $\ag$, one
  can use lemma \ref{jumpin} (by jumping over $v$) to turn the
  coefficient $\beta_u$ into $1$ and replace the coefficients $\beta_w$
  by $\beta_w / \beta_u$, for all vertices at the end of an arrow $u \to w$.

  One can see that the graph $\ag$ has no oriented cycle, otherwise
  there would be a cycle in $T$ made of concatenated
  dominoes. Moreover, edges in the graph $\ag$ can only go
  from green to green, from orange to orange or green, or start from
  red.

  Then one can do these jumps starting from the sources in $\ag$ and then
  proceeding along any linear extension of the partial order defined
  by $\ag$.

  At the end of this process, all vertices covered by dominoes have
  coefficient $1$. There only remains coefficients on the red vertices
  not covered by the maximum matching $M$. This proves the existence of
  the required parameters $\beta$.  

  The fact that the coefficients $\beta_j$ are products of
  coefficients $\alpha_i$ and their inverses is immediate from the
  definition of jumping. 

  Let us now prove uniqueness. Assume there are two such sets of
  parameters $\beta$ and $\bar{\beta}$. Let $x$ and $\bar{x}$ be the
  coordinates on the isomorphic $X_T(\beta)$ and $X_T(\bar{\beta})$.
  
  Let us first prove that any isomorphism in the groupoid $G_T$ from
  $X_T(\beta)$ to $X_T(\bar{\beta})$ maps $\bar{x}_j$ to $x_j$ for
  every green vertex $j$. This is done by induction using the
  auxiliary graph $\ag$, starting with the green vertices that
  do not have any outgoing edge in $\ag$. For every green
  vertex, one just has to consider the equation \eqref{exchange_rel}
  for the unique red vertex that is in the same domino in $M$.

  Using then the equation \eqref{exchange_rel} for all red vertices
  $i$ not covered by $M$, one obtains that $\beta_i =
  \bar{\beta}_i$. This proves uniqueness.

  For the statement $(a)$, consider what happens to the
  coefficient attached to an orange or a green vertex $u$. By
  proposition \ref{precise_matching}, the domino containing $u$ must
  be orange or green-red. The coefficient can therefore only jump to
  green or orange vertices. So they must disappear at some point,
  because only red vertices bear coefficients at the end of the
  process.

  Similarly for the statement $(b)$, consider the coefficient attached to a
  red vertex $u$. Again by proposition \ref{precise_matching}, the
  domino containing $u$ must be red-green. The coefficient can only
  jump to red vertices in the same red-green component, or to orange
  and green vertices. As the coefficients on orange or green vertices
  will disappear by the previous point, coefficients can only stay
  within a given red-green component. 
\end{proof}

Recall that the dimension $\dim(T)$ of $T$ is (by lemma
\ref{dim_is_not_covered}) the number of red vertices that are not
covered in any maximum matching of $T$. Proposition
\ref{reduction_via_matching} justifies this terminology, as this gives
the number of independent parameters for the varieties $X_T(\alpha)$
(inside the groupoid $G_T$).

\begin{remark}
  In the particular case when the tree $T$ is \orange, all
  $X_T(\alpha)$ are isomorphic.
\end{remark}


\medskip

By proposition \ref{reduction_via_matching}, in order to study all
isomorphism classes of such varieties, one can restrict oneself to
attach parameters only to red vertices not covered by a maximum
matching $M$.

For a maximum matching $M$ of $T$, let us define a scheme
$X^M_T(\alpha)$ by the set of equations \eqref{exchange_rel}, where
$\alpha_i$ are invertible fixed parameters, equal to $1$ if $i$ is
covered by $M$.

Given two matchings $M$ and $M'$, one can always find by Proposition
\ref{reduction_via_matching} a sequence of jumps that provides an
isomorphism in $G_T$ between $X^M_T(\alpha)$ and $X^{M'}_T(\beta)$,
where the parameters $\beta$ are uniquely determined Laurent monomials
in $\alpha$.

Let us consider now the automorphism group $\aut(X^M_T(\alpha))$ of
the object $X^M_T(\alpha)$ in the groupoid $G_T$.

\begin{proposition}
  \label{description_aut}
  The automorphism group $\aut(X^M_T(\alpha))$ is an algebraic torus
  isomorphic to $\gm^{\dim(T)}$. If $(\lambda_i)_{i\in T}$ is an element of
  $\aut(X^M_T(\alpha))$, then $\lambda_i = 1$ on green and orange
  vertices of $T$.
\end{proposition}
\begin{proof}
  Let us consider an automorphism in $G_T$ given by invertible
  elements $\lambda_i$.

  The condition that the equation \eqref{exchange_rel} for the vertex $i$ is
  preserved is
  \begin{equation}
    \label{kernel_eq}
    \prod_{j - i} \lambda_j = 1.
  \end{equation}

  This just means that the $\lambda_i$ belongs to the kernel of the
  adjacency matrix of $T$ (seen as an endomorphism of
  $\gm^T$). Looking at the induced linear equations on the tangent
  space at one, one can deduce from remark \ref{rem_kernel} that the
  dimension of $\aut(X^M_T(\alpha))$ is $\dim(T)$.

  By the same argument (using induction on the auxiliary graph
  $\ag$) as in the uniqueness step of the proof of Prop.
  \ref{reduction_via_matching}, every automorphism fixes $x_j$ for
  every green vertex $j$.

  By a similar argument (starting with orange vertices attached to
  green vertices in the auxiliary graph $\ag$), one can then
  prove that every automorphism fixes $x_j$ for every orange vertex
  $j$.

  There remains to show that $\aut(X^M_T(\alpha))$ is connected. Let
  us prove that, given any choice for the values of $\lambda_i$ for $i
  \not \in M$, there is a unique element of $\aut(X^M_T(\alpha))$
  extending this choice.

  This is once again done by induction using the auxiliary graph
  $\ag$. Let us consider a red vertex $j$ that is pointing in $\ag$
  only toward vertices with known $\lambda$. Then there is a unique
  way to fix the value $\lambda_j$ such that \eqref{kernel_eq} holds
  for the green vertex $i$ in the domino of $j$.

  This proves that the kernel is isomorphic to $\gm^{\dim(T)}$.
\end{proof}

Note that the torus $\aut(X^M_T(\alpha))$ and its action on
$X^M_T(\alpha)$ do not depend on $\alpha$. This action therefore
extends to varieties defined as the union of $X^M_T(\alpha)$ over some
family of parameters $\alpha$.

The torus $\aut(X^M_T(\alpha))$ can be written as a product of several
tori, indexed by the red-green components. Every factor acts only on
the red vertices inside a fixed red-green component. This
factorization will be useful later to describe free actions on some varieties.

\subsection{Genericity}

\label{generic_section}

A non-empty set $S$ of \red vertices in a \red-\green component $C$ is
called an \textbf{admissible set} if every \green vertex in $C$ has
either $0$ or $2$ neighbors in $S$.

\begin{lemma}
  Given a \red vertex $u$ in $C$, there is an admissible set containing $u$.
\end{lemma}
\begin{proof}
  One can build an admissible set $S$ starting from $\{u\}$ by repeated
  addition of red vertices. If there is a green vertex $v$ with
  exactly one red neighbor in $S$, then add to $S$ one of the other red
  neighbors of $v$. Repeat until the set $S$ is admissible.
\end{proof}

Let us now introduce an explicit \textbf{genericity condition} on the
parameters attached to a given \red-\green component $C$.

\begin{framed}
  For every admissible set $S$ of \red vertices of $C$, the alternating
  product 
  \begin{equation}
    \label{condition_generique}
    \prod_{i \in S} \alpha_i^{\pm} \not = (-1) ^{\#S},
  \end{equation}
  where any two \red vertices sharing a common \green
  neighbor have opposite powers in the left hand side.
\end{framed}



\smallskip

\begin{lemma}
  \label{generic_and_jump}
  The genericity condition is preserved under jumping moves.
\end{lemma}
\begin{proof}
  Indeed, consider the jumping move from a red vertex $u$ over a
  green vertex $v$. The coefficients of all red neighbors of $v$ are
  divided by $\alpha_u$. Let $S$ be an admissible set. If the vertex
  $v$ has no neighbor in $S$, nothing is changed in the genericity
  condition for $S$. Otherwise, the vertex $v$ has two neighbors in
  $S$. Then two terms are changed in the left-hand side of
  \eqref{condition_generique}, both being divided by $\alpha_u$. But
  they appear with opposite powers, hence the product is not changed.

  The two other kinds of jumping moves (green over red and orange over
  orange) do not change the parameters of red vertices.
\end{proof}


\subsection{Definition of the varieties}

Let us now carefully define the varieties that will be studied in the
rest of the article.

Let us fix a tree $T$, a choice function $\phi$ from the set of red-green
components of $T$ to the set $\{\gen, \ver\}$ and a maximum matching $M$ of
$T$.

For every \red-\green component $C$ such that $\phi(C)$ is $\gen$, let
us fix for every vertex $u$ of $C$ not covered by the maximum matching
$M$, an invertible value $\alpha_u$.

To this data, one associates a scheme $X^{\phi,M}_{T, \alpha}$ as
follows.

The variables are 
\begin{itemize}
\item $x_i$ and $x'_i$ for all vertices of $T$,
\item $\alpha_i$ for all vertices not covered by the matching $M$ in
  the red-green components $C$ of $T$ such that $\phi(C)$ is $\ver$.
\end{itemize}

The equations are 
\begin{itemize}
\item the system of equations \eqref{exchange_rel},
\item all variables $\alpha_i$ are invertible.
\end{itemize}

In fact, there is no true dependency on the matching $M$. Let us
consider two maximum matchings $M$ and $M'$. Using proposition
\ref{reduction_via_matching}, one can find an isomorphism between
$X^{\phi,M}_{T,\alpha}$ for arbitrary invertible parameters $\alpha$
and $X^{\phi,M'}_{T,\beta}$ for parameters $\beta$ depending on the
parameters $\alpha$.

One will therefore forget the matching and use the notation $X^\phi_T$
from now on, keeping the parameters $\alpha$ implicit as well.

\smallskip

Moreover, by lemma \ref{generic_and_jump}, if the genericity condition
\eqref{condition_generique} holds for the parameters $\alpha$ with
respect to one matching $M$, they will also hold for the corresponding
parameters $\beta$ for another matching $M'$.

One can therefore impose that the genericity condition
\eqref{condition_generique} holds for all $\gen$ red-green components
of $T$. This will always be assumed from now on.

\smallskip

Let us summarize this lengthy definition. Once the tree $T$ is chosen,
one picks a maximum matching $M$ of $T$. Any choice of matching will
lead to isomorphic varieties. One then decides for every red-green
component of $T$ either to take the union over all invertible
parameters or to fix some generic parameters.

\smallskip

One will use the simplified notation $X_T$ for orange trees,
as there is then no choice to be made for the function $\phi$. One
will also use the notations $X^\gen_T$ and $X^\ver_T$ when the
function $\phi$ is constant.

\begin{remark}
  One can as well consider forests instead of trees in the definition
  of the varieties $X^\phi_T$, but then everything factors according
  to the connected components. This possibility will be used
  implicitly in the rest of the article.
\end{remark}

\smallskip

Let us introduce the notation $U(x)$ for the open set defined by $x
\not= 0$.

\begin{lemma}
  \label{cover2}
  If $a-b$ is an edge in a tree $T$, then the two open sets $U(x_a)$
  and $U(x_b)$ cover the variety $X^\phi_T$.
\end{lemma}
\begin{proof}
  This follows from the exchange relation
  \begin{equation*}
    x_a x'_a = 1 + \alpha_a x_b y,
  \end{equation*}
  where $y$ is some product of other cluster variables.
\end{proof}

\begin{remark}
  \label{induced_phi}
  When removing red vertices or green vertices in a tree $T$, some
  red-green components may split into several red-green
  components. One can then define a function $\widehat{\phi}$ on the
  new set of red-green components, whose value on a red green
  component $C$ is the value of $\phi$ in the unique red-green
  component of $T$ containing $C$. Abusing notation, one will denote
  this induced function $\widehat{\phi}$ simply by $\phi$.
\end{remark}

\section{Smoothness and free actions}

\begin{theorem}
  \label{main_smooth}
  For every choice of $\phi$, the variety $X^\phi_T$ is smooth.
\end{theorem}

\begin{proof}
  The proof is by induction on the size of the tree $T$.

  For the tree with only one vertex, the only equation is
  \begin{equation}
    \label{typeA1}
    x x' = 1 + \alpha.
  \end{equation}
  
  In the $\gen$ case when $\alpha$ is considered to have a fixed
  value, different from $-1$ by the genericity condition
  \eqref{condition_generique}, the variety is isomorphic to the
  punctured affine line $\gm$ and is therefore smooth.

  In the $\ver$ case when $\alpha$ is considered to be a variable and
  assumed to be invertible, the variety is an open set in the variety
  defined by \eqref{typeA1} where $\alpha$ is not assumed to be
  invertible. This last variety is isomorphic to the affine plane
  $\aff{2}$, hence smooth.

  The rest of the proof by induction is organized as follows. One
  first considers the case when the tree has at least one \red-\green
  component, and treat separately the case when there is a \red-\green
  component which is $\gen$ and the case when there is one which is
  $\ver$. Otherwise, the tree is \orange. These three cases are done
  in the next three subsections.
\end{proof}

Let us first state a few useful lemmas.

\begin{lemma}
  If one variable $x_i$ is assumed to be non-zero, then one can get
  rid of the associated variable $x'_i$ and of the equation
  \eqref{exchange_rel} of index $i$.
\end{lemma}
\begin{proof}
  Indeed, one can just use the equation to eliminate $x'_i$.
\end{proof}

\begin{lemma}
  If one variable $x_i$ is assumed to be zero, then $x'_i$ becomes a
  free variable and the equation \eqref{exchange_rel} of index $i$
  reduces to
  \begin{equation*}
    -1 = \alpha_i \prod_{i-j} x_j.
  \end{equation*}
\end{lemma}

\medskip

Let us now introduce a useful variant of the varieties $X_T^\phi$. Let
$v$ be a vertex of $T$. Let $X^\phi_T[v]$ be defined just as
$X_T^\phi$, but with one more invertible variable $\gamma_v$ attached
to the vertex $v$ as a coefficient (playing the same role as
$\alpha_v$ in the equations). This variable defines a morphism
$\gamma_v$ from $X^\phi_T[v]$ to $\gm$.

\begin{lemma}
  \label{iso_over_gm_orange_green}
  If $v$ is an orange or green vertex, then $X^\phi_T[v]$ is
  isomorphic as a variety over $\gm$ to $X_T^\phi \times \gm$ endowed
  with the projection to the second factor.
\end{lemma}
\begin{proof}
  By proposition \ref{reduction_via_matching} and its proof, one can
  find an isomorphism in the groupoid $G_T$ between $X^\phi_T$ and
  $X^\phi_T[v]$ that only changes the coordinates $x_i$ for orange and
  green vertices. More precisely, using the auxiliary oriented graph
  $\ag$, one can find a sequence of jumps (corresponding to edges in
  $\ag$ starting with a green or orange vertex) that makes the
  coefficient $\gamma_v$ disappear from the equations.

  The isomorphism associated with this sequence of jumps is
  multiplying the variables $x_i$ by monic Laurent monomials in the
  parameter $\gamma_v$, hence defines an isomorphism over $\gm$.
\end{proof}

\begin{lemma}
  \label{iso_over_gm_versal}
  If $v$ is a red vertex in a versal red-green component $C$, then
  $X^\phi_T[v]$ is isomorphic as a variety over $\gm$ to $X_T^\phi
  \times \gm$ endowed with the projection to the second factor.
\end{lemma}
\begin{proof}
  If the red vertex $v$ is not covered by the matching $M$ chosen to
  define $X^\phi_T$, then one has two coefficient variables $\alpha_v$
  and $\gamma_v$ attached to the vertex $v$. By the simple change of
  coordinates $\alpha_v := \alpha_v \gamma_v$ and $\gamma_v := \gamma_v$,
  one gets the expected isomorphism.

  Assume now that that red vertex $v$ is covered by the matching $M$.

  By proposition \ref{reduction_via_matching} and its proof, one can
  find an isomorphism in the groupoid $G_T$ between $X^\phi_T[v]$ and
  a variety $X^{\phi,M}_{T,\beta}$ that only changes the coordinates $x_i$
  for orange and green vertices and for red vertices in the red-green
  component $C$. More precisely, using the auxiliary oriented graph
  $\ag$, one can find a sequence of jumps that moves the coefficient
  $\gamma_v$ towards the red vertices in $C$ not covered by the
  matching. At the end, every new coefficient $\beta_i$ is the product
  of $\alpha_i$ by a Laurent monomial in $\gamma_v$.

  The isomorphism associated with this sequence of jumps is
  multiplying the variables $x_i$ by monic Laurent monomials in the
  parameter $\gamma_v$, hence defines an isomorphism over $\gm$. One
  can then compose this isomorphism with a relabeling of the
  coefficients $\alpha_i := \beta_i$ in order to get the expected
  isomorphism, still defined over $\gm$, between $X^\phi_T[v]$ and
  $X_T^\phi \times \gm$.
\end{proof}

One could say that the coefficient $\gamma_v$ can be \textbf{detached} from $T$
in these cases. This will be used frequently in the rest of the
article.

\subsection{Trees with a generic component}

\label{generic_is_smooth}

One assumes now that $T$ has at least two vertices and a $\gen$
component $C$.

Let us pick an admissible set $S$ of red vertices in $C$, as defined in \S \ref{generic_section}.



\begin{lemma}
  \label{admi_cover}
  The open sets $U(x_i)$ for $i \in S$ form a covering of $X^\phi_T$.
\end{lemma}
\begin{proof}
  Indeed, the complement of their union is the set where all variables
  $x_i$ for $i \in S$ vanish. This implies that
  \begin{equation}
    \alpha_i \prod_{j-i} x_j = -1
  \end{equation}
  for every $i$ in $S$. Taking the alternating product of these
  equalities gives
  \begin{equation}
    \label{nongen}
    \prod_{i \in S} \alpha_i^{\pm} = (-1)^{\# S},
  \end{equation}
  because for every green vertex $j$ attached by an edge to some
  element of $S$, the cluster variable $x_j$ appears exactly twice by
  definition of admissible sets, hence disappears in the alternating
  product.

  But the equation \eqref{nongen} is incompatible with the genericity
  condition \eqref{condition_generique}.
\end{proof}

Let us now show that the open sets $U(x_i)$ are smooth.

Let $F$ be the forest $ T \setminus \{i\}$. In the forest $F$, the
coloring is changed only on the red-green component containing $i$,
where an avalanche of orange dominoes can take place when removing
$i$. The red-green component $C$ is therefore split into a number of
red-green components. Let us moreover introduce a function $\phi$ on
$F$, which is $\gen$ on every red-green component coming from $C$, and
unchanged on all other red-green components.

\begin{lemma}
  \label{smooth_generic}
  The open set $U(x_i)$ is isomorphic to $\gm \times X^\phi_F$.
\end{lemma}
\begin{proof}
  The condition that $x_i$ is not zero allows one to get rid of the
  variable $x'_i$ by using the equation \eqref{exchange_rel} of index
  $i$. What remains are the equations for the forest $F = T \setminus
  \{i\}$, where now $x_i$ is treated as a parameter attached to all
  neighbors of $i$ in $T$.

  Because all neighbors of $i$ in $T$ are green, they become either
  green or orange in $F$. It follows from lemma
  \ref{iso_over_gm_orange_green} that one can, without changing the
  variety, consider instead that the parameter $x_i$ is not attached
  to any vertex of $F$.

  Let us check that the genericity condition still holds on all $\gen$
  red-green components. If the component $D$ does not come from the
  splitting of $C$, then the genericity conditions are unchanged on
  this red-green component. Otherwise, let us choose an admissible set
  in $D$. It was then already an admissible set in $C$, by inspection
  of what happens during the avalanche of orange dominoes. Therefore
  the genericity condition for $D$ is inherited from that for $C$.
\end{proof}

One has therefore obtained an isomorphism
\begin{equation}
  U(x_i) \simeq \gm \times X^\phi_F,
\end{equation}
which is smooth by induction. Therefore $X^\phi_T$ is also smooth.

\subsection{Trees with a versal component}

One assumes now that $T$ has at least two vertices, and has a $\ver$
component $C$. Let us choose a red leaf $v$ in this component. By
proposition \ref{rouge_dans_matching}, one can find a maximum matching
$M$ not containing $v$. Therefore there is a coefficient variable
$\alpha_v$.


Let $u$ be the \green vertex adjacent to $v$. By lemma \ref{cover2},
the two open sets $U(x_u)$ and $U(x_v)$ cover $X^\phi_T$.

\smallskip

Let us first prove that $U(x_v)$ is smooth.

Let $T'$ be the tree $T\setminus \{v\}$. The coloring of $T'$ is obtained
from $T$ by an avalanche of orange dominoes. The dimension of $T'$ is
$\dim(T) - 1$.

The avalanche may split the red-green component of $T$ containing
$v$ into several components. Let $\phi$ be the function which maps all
these new components to the $\ver$ condition, and unchanged condition
on all the other red-green components.

\begin{lemma}
  \label{smooth_versal_red}
  The open set $U(x_v)$ is isomorphic to $\gm^2 \times X^\phi_{T'}$.
\end{lemma}
\begin{proof}
  Assuming that $x_v$ is not zero allows one to get rid of the
  variable $x'_v$ by using \eqref{exchange_rel} with index $v$. The
  coefficient variable $\alpha_v$ also disappears from the equations:
  this gives one factor $\gm$.

  Then the variable $x_v$ is seen as a coefficient attached to the
  vertex $u$ in $T'$, which is either green or orange. The coefficient
  can therefore be detached by lemma \ref{iso_over_gm_orange_green},
  and one obtains a factor isomorphic to $\gm \times X^\phi_{T'}$.
\end{proof}

Therefore $U(x_v)$ is smooth by induction.

\smallskip

Let us now prove that $U(x_u)$ is smooth. Let us choose instead a
matching $M$ containing the domino $u-v$, thanks to lemma
\ref{presque_pavage}. This amounts to go through an isomorphism in
the groupoid $G_T$, hence preserves the open set $U(x_u)$.

Let $F$ be the forest $T\setminus \{u\}$. Because $u$ is green, the
coloring of $F$ is obtained from that of $T$ by restriction and the
dimension of $F$ is $\dim(T) + 1$. Let $v, T_1, \dots, T_k$ be the
connected components of the forest $F$. By removing the domino $u-v$,
one can restrict the matching $M$ to a matching of the forest $F$.

The red-green component of $T$ containing $u$ splits into several
red-green components in $F$, one of them being the vertex $v$. One
takes the $\ver$ condition on all of these red-green components of
$F$, and unchanged condition on all the other red-green components.

\begin{lemma}
  \label{smooth_versal_green}
  The open set $U(x_u)$ is isomorphic to 
  \begin{equation}
    X^{\ver}_{\{v\}} \times \prod_{j=1}^k X^\phi_{T_j},
  \end{equation}
  where the first component is the vertex $v$ with coefficient variable $x_u$.
\end{lemma}
\begin{proof}
  Setting $x_u\not =0$ in the equations allows to get rid of the
  variable $x'_u$. The result can be described as a fiber product over
  $\gm$, where the same coefficient variable $x_u$ is attached to
  every connected component of $F$ at a red vertex in a versal
  red-green component. By repeated use of lemma
  \ref{iso_over_gm_versal} on all connected components (but not on
  the isolated vertex $v$), one finds that the open set $U(x_u)$ is isomorphic to the product
  \begin{equation}
    X^{\ver}_{v} \times \prod_{j=1}^k X^\phi_{T_j},
  \end{equation}
  where the first component is the vertex $v$ with coefficient $x_u$.
\end{proof}

Therefore $U(X_u)$ is smooth by induction, and hence $X^\phi_T$ is also smooth.

\subsection{Orange trees}

Let us now assume that $T$ is an \orange tree and let us choose one
domino $u-v$ in the perfect matching of $T$. By lemma \ref{cover2},
the two open sets $U(x_u)$ and $U(x_v)$ cover the variety $X_T$.

By symmetry between $u$ and $v$, it is enough to prove that $U(x_u)$ is
smooth.

Let $T_1, \dots, T_k$ be the trees attached to $u$ in $T\setminus
\{v\}$. The $T_i$ are clearly \orange trees.

Let $R$ be the connected component of $v$ in $T\setminus \{u\}$. The
tree $R$ is obtained by removing a leaf in an \orange tree, hence (by
lemma \ref{dimension1}) has dimension $1$ and a unique \red-\green
component. Moreover, $R$ has a maximum matching avoiding only $v$ and
the vertex $v$ is \red in the coloring of $R$.

\begin{lemma}
  \label{smooth_orange}
  The open set $U(x_u)$ is isomorphic to the product of the varieties
  $X_{T_i}$ and the variety $X^\ver_R$.
\end{lemma}
\begin{proof}
  Assuming that $x_u$ is not zero allows to eliminate the variable
  $x'_u$ and the equation \eqref{exchange_rel} of index $u$.

  There remains the equations for the union of $R$ and the $T_i$, with
  $x_u$ considered as a parameter attached to all of them at the
  former neighbors of $u$.

  Because the trees $T_i$ are \orange, one can consider instead (by
  lemma \ref{iso_over_gm_orange_green}) that the parameter $x_u$ is
  only attached to the vertex $v$ of $R$.

  This proves that the open set $U(x_u)$ is isomorphic to the product
  of the varieties $X_{T_i}$ and the variety $X^\ver_R$.
\end{proof}

By induction, this proves that $U(x_u)$ is smooth. Therefore $X_T$ is
smooth too.

\subsection{Torus actions}

\label{torusaction}

Let $T$ be a tree and let $\phi$ be a choice in $\{\gen,\ver\}$ for
every red-green component of $T$. Let us also choose a maximum
matching $M$ of $T$.

One can deduce from proposition \ref{description_aut} and the remarks
following it that there is an action of an algebraic torus of
dimension $\dim(T)$ on $X^\phi_T$, and that this torus (and its
action) can be written as a product over red-green components $C$ of
tori $\Lambda^C_T$.

Let us define a smaller torus $\Lambda^\phi_T$ acting on $X^\phi_T$ as
the product of $\Lambda^C_T$ over all $\gen$ red-green components of
$T$. Let us call the \textbf{rank} of $(T, \phi)$ and denote by
$\rk(T,\phi)$ the sum of the dimensions of the generic red-green
components of $T$. This is the dimension of $\Lambda^\phi_T$.

\begin{proposition}
  If $\phi(C)$ is generic, the action of $\Lambda^C_T$ on $X^\phi_T$ is
  free.
\end{proposition}
\begin{proof}
  Let us assume that there is a non-trivial element $\lambda=(\lambda_i)_i$ of
  $\Lambda^C_T$ that fixes a point $(x_i)_i$ in $X^\phi_T$.

  Let $i$ be a red vertex in $C$ such that $\lambda_i \not= 1$. For
  every green neighbor $j$ of $i$, one can find another red vertex $k$
  incident to $j$ such that $\lambda_k \not = 1$, because of
  \eqref{kernel_eq}. Iterating this process, one can build an
  admissible set $S$ (as defined in \S \ref{generic_section}), such
  that $\lambda_s \not=1$ for every $s \in S$.

  Because $\lambda$ fixes the given point, one then has $x_s = 0$ for
  every $s \in S$. But this is impossible by Lemma \ref{admi_cover}.
\end{proof}

\begin{corollary}
  \label{maxifree}
  There is on $X^\phi_T$ a free action by a torus $\Lambda^\phi_T$ of
  dimension the rank $\rk(T,\phi)$.
\end{corollary}

This gives $X^\phi_T$ the structure of a principal bundle with
structure group $\Lambda^\phi_T$. As one will see later, this bundle
is not trivial in general (\textit{i.e.} not a product), as can be
seen from our results for the cohomology already in type $\TA_3$.


\section{Number of points over finite fields and Euler characteristic}

Let us denote by $N^\phi_T(q)$ the number of points on $X^\phi_T$ over
the finite field $\fq$.

When the tree is orange, one will use the shorthand notation
$N_T$. When the function $\phi$ is constant, one will use the
notations $N^\ver_T$ and $N^\gen_T$.

\begin{proposition}
  The numbers $N^\phi_T(q)$ are monic polynomials in $q$ of degree
  $\dim X_T^\phi$.
\end{proposition}
\begin{proof}
  The proof is by induction on the size of the tree.

  For the tree with one vertex, the number of points is $q-1$ in the
  $\gen$ case and $q^2 - q + 1$ in the $\ver$ case, by the description
  given at the beginning of the proof of theorem \ref{main_smooth}.

  Then either the tree has a red-green component, which can be $\gen$
  or $\ver$, or it is an orange tree. The proof is decomposed into the
  three following geometric decomposition lemmas, or rather into their obvious
  corollaries on the number of points over finite fields.
\end{proof}

Let $T$ be a tree and $v$ be a red leaf in a red-green component $C$
of $T$. Let $u$ be the neighbor of $v$. Removing the vertex $v$
creates an orange avalanche and may separate the red-green component
$C$ into several ones. Let $\phi$ be the induced genericity condition
(as defined in Remark \ref{induced_phi}). Let $F$ be the forest $T
\setminus \{u, v\}$. The component $C$ may also split into several
red-green components in $F$. Let $\phi$ be the induced genericity
condition.

\smallskip

Let us consider now the case of a generic red-green component $C$.
\begin{lemma}
  \label{number_generic}
  In this situation, the variety $X^\phi_T$ can be decomposed as
  \begin{equation}
    X^\phi_T  = \gm  X^\phi_{T\setminus
      \{v\}} \sqcup \aff{1} X^\phi_F.
  \end{equation}
\end{lemma}
\begin{proof}
  Either $x_v$ is not zero or $x_v$ is zero. This will give the
  required disjoint union. In the case where $x_v \not = 0$, one uses
  lemma \ref{smooth_generic}. This gives the first term of the right
  hand side.
  
  Let us pick a maximum matching $M$ of $T$ containing $v$. This is
  possible by lemma \ref{presque_pavage}. This does not change the
  open set $U(x_v)$ and its complement, up to isomorphism.

  Assume now that $x_v$ is zero. Then $x'_v$ is a free variable, and
  $x_u$ is equal to $-1$, because there are no coefficients on
  $v$. One then gets rid of $x'_u$. The coloring of the forest $F$ is
  by restriction of the coloring of $T$. Therefore the parameter $x_u
  = -1$ is attached to some red vertices of $F$, as a coefficient.

  One has to check that the genericity condition still holds on every
  connected component of $F$. Let $S$ be an admissible set in one of
  these components. Either $S$ was already an admissible set in $T$,
  and then the genericity condition still holds, or it contains
  exactly one of the neighbors of $u$ in $T$. In this case, one can
  extend $S$ by adding $v$ to form an admissible set in $T$. The
  genericity condition for $S \sqcup \{v\}$ in $T$ implies the
  condition for $S$, because of the additional $-1$ coefficient
  attached to $S$ in $F$.
\end{proof}

Keeping the same notations, let us consider now the case of a versal
red-green component $C$.
\begin{lemma}
 \label{number_versal}
  In this situation, the variety $X^\phi_T$ can be decomposed as
  \begin{equation}
    X^\phi_T = \gm^2 X^\phi_{T\setminus \{v\}} \sqcup \aff{1} X^\phi_F.
  \end{equation}
\end{lemma}
\begin{proof}
  Either $x_v$ is not zero or $x_v$ is zero. This will give the
  required disjoint union. If $x_v \not =0$, using lemma
  \ref{smooth_versal_red} gives the first term of the right hand side.
  
  Let us pick a maximum matching $M$ of $T$ containing $v$. This is
  possible by lemma \ref{presque_pavage}. This does not change the
  open set $U(x_v)$ and its complement, up to isomorphism.

  Assume now that $x_v$ is zero. Then $x'_v$ is a free variable, and
  $x_u$ is equal to $-1$, because there are no coefficients on
  $v$. One then gets rid of $x'_u$. The coloring of the forest $F$ is
  by restriction of the coloring of $T$. Therefore the parameter $x_u = -1$
  is attached to red vertices of $F$. By lemma
  \ref{iso_over_gm_versal}, it can be detached, and this just gives
  the expected second term.
\end{proof}

Let $T$ be an orange tree and $u-v$ be a domino in $T$. Let
$(T_{u,i})_i$ (resp. $(T_{v,j})_j$) be the connected components of $T
\setminus \{u,v\}$ that were attached to $u$ (resp. to $v$). All these
trees are orange. Let us denote by $S_{u,i}$ and $S_{v,j}$ the forests
obtained from them by removing the vertex that was linked to $u$ or
$v$. These forests are unimodal, in the sense that they have one
unimodal connected component, all the other connected components being
orange.

\begin{lemma}
  \label{number_orange}
  In this situation, one has
  \begin{equation*}
    X_T = \gm^2 \prod_i X_{T_{u,i}} \prod_j X_{T_{v,j}}  \sqcup \aff{1} \prod_i X^\ver_{S_{u,i}}\prod_j X_{T_{v,j}} \sqcup \aff{1} \prod_i X_{T_{u,i}}\prod_j X^\ver_{S_{v,j}}  .
  \end{equation*}
\end{lemma}
\begin{proof}
  Because the open sets $U(x_u)$ and $U(x_v)$ are a covering by lemma
  \ref{cover2}, one can cut the variety $X_T$ into three pieces:
  either both $x_u$ and $x_v$ are not zero, or exactly one of them is
  zero.

  If both are not zero, then one obtains the product of $\gm^2$ (with
  coordinates $x_u$ and $x_v$) with the product of the varieties
  attached to the $T_{u,i}$ and the $T_{v,j}$. Indeed, one first get
  that $x_u$ becomes a parameter attached to all trees $T_{u,i}$ and
  $x_v$ becomes a parameter attached to all trees $T_{v,j}$. But these
  trees are orange, so $x_u$ and $x_v$ can be detached by lemma
  \ref{iso_over_gm_orange_green}. This gives the first term.

  If $x_u$ is zero and $x_v$ is not zero, then there is a free
  variable $x'_u$ and the variable $x_v$ is determined by the
  variables attached to the vertices of the trees $T_{u,i}$ linked to
  $u$, which must be non-zero. One obtains therefore a versal
  condition on each forest $S_{u,i}$. For the trees $T_{v,j}$, the
  coefficient $x_v$ is attached to all of them, but because they are
  orange it can be detached. This gives the second term.

  The third term is the same after exchanging $u$ and $v$.
\end{proof}


\subsection{Reciprocal property}

Recall from \S \ref{torusaction} that the rank $\rk(T,\phi)$ of the
pair $(T,\phi)$ formed by a tree $T$ and a choice function $\phi$ is
the sum of the dimensions of the $\gen$ red-green components of $T$.

\begin{proposition}
  The polynomial $N^\phi_T(q)$ is divisible par $(q-1)^{\rk(T,\phi)}$.
\end{proposition}
\begin{proof}
  This follows from the existence of the free action obtained in
  corollary \ref{maxifree}.
\end{proof}

Let us refine this slightly.

\begin{proposition}
  The polynomial $N_T^\phi$ can be written as $(q-1)^{\rk(T,\phi)}$ times a
  reciprocal polynomial.
\end{proposition}
\begin{proof}
  By induction. This is true for the tree with one vertex.

  One just has to look carefully at the decompositions given in the
  three lemmas that were used to prove polynomiality by induction.

  For lemma \ref{number_generic}, let $D$ be the rank for $T$. Then
  the rank is $D-1$ for $T \setminus \{v\}$ and $D$ for $F$. Using the
  additional factor $q-1$ coming from $\gm$, there is a common factor
  $(q-1)^D$ to all terms involved. The factor $\aff{1}$ in the
  codimension $1$ piece ensures that the reciprocal property holds.

  For lemma \ref{number_versal}, the rank $D$ is the same in all terms
  involved. One uses that $(q-1)^2$ is reciprocal. The factor
  $\aff{1}$ in the codimension $1$ piece ensures that the reciprocal
  property holds.

  For lemma \ref{number_orange}, the rank $D$ is $0$ in all terms
  involved, as there is no generic red-green component. One uses again
  that $(q-1)^2$ is reciprocal. The factor $\aff{1}$ in the
  codimension $1$ pieces ensures that the reciprocal property holds.
\end{proof}

\subsection{Enumeration and coincidences}

In the following remarks, one will describe trees by their numbers in
the tables at the end of \cite{spectra_book} and by their graph6
string (which is a standard format for graphs).

\begin{remark}
  One can find distinct \orange trees with the same enumerating
  polynomial. This happens first for trees with $8$ vertices. The
  trees 2.188 (graph6 \verb!'IhGGOC@?G'!) and 2.189 (graph6
  \verb!'IhC_GCA?G'!) have the same polynomial, as well as the trees
  2.172 (graph6 \verb!'IhGGOCA?G'!) and 2.174 (graph6
  \verb!'IhGH?C@?G'!). The number of different polynomials for \orange
  trees with $2n$ vertices is the sequence
  \begin{equation*}
    1, 1, 2, 5, 13, 41, 138, \dots
  \end{equation*}
  whereas the number of \orange trees is
  \begin{equation*}
    1, 1, 2, 5, 15, 49, 180, \dots
  \end{equation*}
\end{remark}

\begin{remark}
  For unimodal trees with $\ver$ condition, one can also find pairs
  with the same enumerating polynomials. The smallest one is made of
  trees with $9$ vertices, numbered 2.83 (graph6 \verb!'HhCGOCA'!) and
  2.85 (graph6 \verb!'HhGGGG@'!). The number of different polynomials
  for unimodal trees with $2n+1$ vertices is the sequence
  \begin{equation*}
    1, 1, 2, 6, 19, 65, \dots
  \end{equation*}
  whereas the number of unimodal trees is
  \begin{equation*}
    1, 1, 2, 6, 20, 76, 313, 1361, \dots
  \end{equation*}
\end{remark}

\begin{remark}
  For unimodal trees with $\gen$ condition, one can also find pairs
  with the same enumerating polynomials. The smallest one is made of
  the Dynkin diagrams $\TA_7$ and $\TE_7$. The number of different
  polynomials for unimodal trees with $2n+1$ vertices is the sequence
  \begin{equation*}
    1, 1, 2, 5, 13, 46, 168, \dots
  \end{equation*}
\end{remark}




\subsection{Linear trees}
\label{pointsA}

Let us denote by $\TA_n$ the linear tree with $n$ vertices.




\begin{center}
{\scalefont{0.7}
\begin{tikzpicture}[scale=0.7]
\tikzstyle{every node}=[draw,shape=circle,very thick,fill=white]

\draw (0,0) node[fill=orange!20] {1} -- (1,0) node[fill=orange!20] {2} -- (2,0) node[fill=orange!20] {3} -- (3,0) node[fill=orange!20] {4} -- (4,0) node[fill=orange!20] {...} -- (5,0) node[fill=orange!20]{$n$};

\end{tikzpicture}
\quad
\begin{tikzpicture}[scale=0.7]
\tikzstyle{every node}=[draw,shape=circle,very thick,fill=white]

\draw (0,0) node[fill=red!20] {1} -- (1,0) node[fill=green!20]{2} -- (2,0) node[fill=red!20] {3} -- (3,0) node[fill=green!20] {4} -- (4,0) node[fill=red!20] {...} -- (5,0) node[fill=green!20]{...} -- (6,0) node[fill=red!20]{$n$};

\end{tikzpicture}}
\end{center}

One can check that $\TA_n$ is orange if $n$ is even and unimodal if $n$ is odd.

\begin{proposition}
  \label{nbrA}
  The number of points on varieties attached to $\TA_n$ is given by
  \begin{equation}
    \label{nbrApair}
    N_{\TA_n} = \frac{q^{n+2} - 1}{q^2 -1}
  \end{equation}
  if $n$ is even and by
  \begin{equation}
     N_{\TA_n}^\ver =  \frac{q^{n+2} + 1}{q + 1} \quad \text{and} \quad  N_{\TA_n}^\gen = \frac{(q^{(n+1)/2} - 1)(q^{(n+3)/2} - 1)}{q^2 -1}
  \end{equation}
  if $n$ is odd.
\end{proposition}
\begin{proof}
  This follows easily by induction from lemmas \ref{number_generic},
  \ref{number_versal} and \ref{number_orange}. 
\end{proof}

\subsection{Trees of type $\TD$}

Let us denote by $\TD_n$ the tree with $n$ vertices associated with the
Dynkin diagram of type $\TD$.

\begin{center}
{\scalefont{0.7}
\begin{tikzpicture}[scale=0.7]
\tikzstyle{every node}=[draw,shape=circle,very thick,fill=white]

\draw (0.3,0.7) node[fill=red!20] {1} -- (1,0) -- (2,0) node[fill=orange!20] {4} -- (3,0) node[fill=orange!20] {5} -- (4,0) node[fill=orange!20] {...} -- (5,0) node[fill=orange!20]{$n$};

\draw (0.3,-0.7) node[fill=red!20] {2} -- (1,0) node[fill=green!20] {3};

\end{tikzpicture}
\quad
\begin{tikzpicture}[scale=0.7]
\tikzstyle{every node}=[draw,shape=circle,very thick,fill=white]

\draw (0.3,0.7) node[fill=red!20] {1} -- (1,0) -- (2,0) node[fill=red!20] {4} -- (3,0) node[fill=green!20] {5} -- (4,0) node[fill=red!20] {...} -- (5,0) node[fill=green!20]{...} -- (6,0) node[fill=red!20]{$n$};

\draw (0.3,-0.7) node[fill=red!20] {2} -- (1,0) node[fill=green!20] {3};

\end{tikzpicture}}
\end{center}

One can check that $\TD_n$ is unimodal if $n$ is odd and has dimension
$2$ if $n$ is even.

\begin{proposition}
  \label{nbrD}
  The number of points on varieties attached to $\TD_n$ is given by
  \begin{equation}
    \label{nbrDpair}
    N_{\TD_n}^\ver =  \frac{q^{n+3}-q^{n+2}+q^{n}+q^3-q+1}{q + 1} \quad \text{and} \quad  N^\gen_{\TD_n} = (q^{n/2} - 1)^2
  \end{equation}
  if $n$ is even and by
  \begin{equation}
     N_{\TD_n}^\ver =  \frac{q^{n+3}-q^{n+2}+q^{n}-q^3+q-1}{q^2-1} \quad \text{and} \quad  N_{\TD_n}^\gen = q^n-1
  \end{equation}
  if $n$ is odd.
\end{proposition}
\begin{proof}
  This is easily deduced from the type $\TA$ case, using \ref{number_generic},
  \ref{number_versal} applied to a red leaf on a short branch.
\end{proof}

\subsection{Trees of type $\TE$}

\label{type_E}

Let us consider now a family of trees containing the Dynkin diagrams
of type $\TE$. The tree $\TE_n$ is the tree with one triple point and
branches of size $1$, $2$ and $n-4$.

\begin{center}
{\scalefont{0.7}
\begin{tikzpicture}[scale=0.7]
\tikzstyle{every node}=[draw,shape=circle,very thick,fill=white]

\draw (0,0) node[fill=orange!20] {1} -- (1,0) node[fill=orange!20] {2} -- (2,0) -- (3,0) node[fill=orange!20] {5} -- (4,0) node[fill=orange!20] {6};

\draw (2,1) node[fill=orange!20] {3} -- (2,0) node[fill=orange!20] {4};

\end{tikzpicture}
\quad
\begin{tikzpicture}[scale=0.7]
\tikzstyle{every node}=[draw,shape=circle,very thick,fill=white]

\draw (0,0) node[fill=orange!20] {1} -- (1,0) node[fill=orange!20] {2} -- (2,0) -- (3,0) node[fill=red!20] {5} -- (4,0) node[fill=green!20] {6} -- (5,0) node[fill=red!20]{7};

\draw (2,1) node[fill=red!20] {3} -- (2,0) node[fill=green!20] {4};

\end{tikzpicture}}
\end{center}

One can check that $\TE_n$ is orange if $n$ is even and unimodal if $n$ is odd.

\begin{proposition}
  \label{nbrE}
  The number of points on varieties attached to $\TE_n$ is given by
  \begin{equation}
    \label{nbrEpair}
    N_{\TE_n} = (q^2 - q + 1)\frac{q^{n-1} - 1}{q-1}
  \end{equation}
  if $n$ is even and by $N_{\TE_n}^\ver = (q^2 - q + 1)(1+q^{n-1})  $ and 
  \begin{equation}
       N_{\TE_n}^\gen = \frac{q^{n+1}-q^{n}+q^{n-1}-q^{(n+3)/2}-q^{(n-1)/2}+q^{2}-q+1}{q-1}
  \end{equation}
  if $n$ is odd.
\end{proposition}
\begin{proof}
  In the even case, one uses lemma \ref{number_orange} applied to the
  domino on the short branch, and the known type $\TA$ cases. In the
  odd case, one uses lemmas \ref{number_versal} and
  \ref{number_generic} applied to the red leaf on the short branch,
  and the known type $\TA$ cases.
\end{proof}


\subsection{Orange trees and unimodal trees}

Let us now describe a recursion involving only the polynomials for
orange trees and versal unimodal trees.

Let $T$ be an orange tree and $v$ be a leaf of $T$. Let $T'$ be the
unimodal tree $T \setminus \{v\}$ and let $F$ be the orange forest
obtained from $T$ by removing the domino $u - v$ containing $v$.

\begin{lemma}
  \label{pour_algo_orange_unimodal}
  There is a decomposition
  \begin{equation}
    X_T = X^{\ver}_{T'} \sqcup \aff{1} X_F.
  \end{equation}
\end{lemma}
\begin{proof}
  This decomposition is made according to the value of $x_v$.

  If $x_v = 0$, then one has a free parameter $x'_v$, which gives the
  factor $\aff{1}$. One also has $x_u = -1$ and one can get rid of
  $x'_u$. The value $-1$ is attached as a coefficient to some orange
  vertices of $F$, but one can detach this coefficient by lemma
  \ref{iso_over_gm_orange_green}. There remains the equations for $X_F$.

  If $x_v \not = 0$, one can use lemma \ref{smooth_orange}. In the
  special case of a leaf, this gives an isomorphism with
  $X^\ver_{T'}$.
\end{proof}

One can use lemma \ref{pour_algo_orange_unimodal} to compute the
enumerating polynomials for orange trees and versal unimodal trees
only, by the following algorithm.

\textbf{Step} 0: if the tree $T$ is of type $\TA_n$ with $n$ even, use
the known value from \eqref{nbrApair} in proposition \ref{nbrA}.

\textbf{Step} 1: if the tree $T$ is orange, find a leaf $v$ whose
branch has minimal length. Here the branch is the longest sequence of
vertices of valency $2$ starting at the unique neighbor of the leaf
(it could be empty). Then use lemma \ref{pour_algo_orange_unimodal}
applied to the leaf $v$ to compute $N_T$.

\textbf{Step} 2: if the tree $T$ is unimodal, find a red leaf $w$
whose branch has maximal length. Adding a vertex $v$ at the end of
this branch gives an orange tree $T'$. Then use lemma
\ref{pour_algo_orange_unimodal} (backwards) applied to the tree $T'$
and its leaf $v$ to compute $N_T$.

This will work because each step either shorten the shortest branch or
add some vertex to the longest branch. This makes sure that the tree
become more and more linear, and that at some point one is reduced to
the initial step. This is a decreasing induction on the number of
points of valency at least $3$ and the length of the longest branch.


\begin{remark}
  For orange trees, one can use instead in this algorithm the
  lemma \ref{number_orange}, maybe choosing a domino close to the
  center of the tree for a better complexity.
\end{remark}


\subsection{Euler characteristic and independent sets}

Let us denote by $\operatorname{vc}(T)$ the number of minimum vertex covers of
$T$. This is also the number of maximum independent sets.

Let us now describe a decomposition of the versal varieties according
to independent sets (not necessarily maximal).

If $S$ is a subset of the vertices of $T$, one can define $W_T(S)$ as
the set of points in $X_T^\ver$ where
\begin{align}
  x_u = 0 & \quad\text{if } u\in S,\\
  x_u \not= 0 & \quad\text{if } u\not\in S.
\end{align}

The sets $W_T(S)$ are obviously disjoint in $X_T^\ver$.

\begin{lemma}
  If the set $W_T(S)$ is not empty, then $S$ is an independent set in $T$.
\end{lemma}
\begin{proof}
  This follows from lemma \ref{cover2}.
\end{proof}

\begin{proposition}
  \label{decoupe_independante}
  Let $S$ be an independent set in $T$. There is an isomorphism
  \begin{equation*}
    W_T(S) \simeq (\gm)^{t + \dim(T)- 2s} \times (\aff{1})^{s},
  \end{equation*}
  where $t$ is the size of $T$ and $s$ the size of $S$.
\end{proposition}
\begin{proof}
  Let us fix a maximum matching $M$ of $T$.

  For every $u$ not in $S$, one can use the hypothesis $x_u \not= 0$
  to get rid of $x'_u$ and of the equation of index $u$. There remains
  only the equations of index $v$ for $v \in S$. Because $x_v=0$ when
  $v \in S$, the variables $x'_v$ for $v \in S$ do no longer appear in
  the equations, hence they are free. This gives the factor
  $(\aff{1})^s$.

  Then there remains $s$ equations of the general shape
  \begin{equation}
    \tag{$E_i$}
    \label{equi}
    -1 = \alpha_i \prod_{j - i} x_j,
  \end{equation}
  involving the $t-s$ invertible variables $x_u$ and the $\dim(T)$
  coefficient variables $\alpha_i$. The factor $\alpha_i$ is present
  in this equation only if the vertex $i$ is not covered by the chosen
  maximum matching $M$.

  One will use the following auxiliary graph $\widehat{T}$. The
  vertices are the vertices of $T$ and new vertices $Z_i$ indexed by
  coefficient variables $\alpha_i$ for $i \not\in M$. The edges of
  $\widehat{T}$ are edges of $T$ and new edges between the vertex
  $Z_i$ and the vertex $i$ for every $i \not \in M$. Clearly, this
  graph is still a tree and admits a perfect matching $\widehat{M}$, by
  adding dominoes $i-Z_i$ to the matching $M$.

  Because $S$ is an independent set in $T$, there is at most one
  element of $S$ in every edge of $\widehat{T}$. Let us orient every
  edge containing an element of $S$ towards this element if the edge
  is a domino and in the other way otherwise. This
  defines a partial order on the vertices of $\widehat{T}$, decreasing
  along the chosen orientation of edges.

  Consider now the equation \ref{equi} associated with a vertex $i \in
  S$. There is a unique domino $i-j$ in $\widehat{T}$ containing
  $i$. The equation can then be used to express the variable $x_j$ in
  terms of variables of lower index in the partial order.

  One can therefore eliminate one variable for every equation.  At the
  end, one obtains an algebraic torus whose dimension is the
  difference between the number $ t -s +\dim(T)$ of initial variables
  and the number $s$ of equations.






\end{proof}

\begin{corollary}
  The Euler characteristic of $X_T^\ver$ is $\operatorname{vc}(T)$.
\end{corollary}
\begin{proof}
  Every set $W_T(S)$ contributes either $0$ or $1$ to the Euler
  characteristic. It contributes by $1$ if and only if the exponent $t
  + \dim(T)- 2s$ is zero.

  This exponent can be expressed as
  \begin{equation*}
    (r(T) + o(T) + g(T)) + (r(T) - g(T)) - 2 s.
  \end{equation*}
  It is therefore zero if and only if $s = r(T) + o(T)/2$, which is the
  size of the maximum independent sets in $T$.
\end{proof}

Of course, one can also use Proposition \ref{decoupe_independante} to
give a formula for the number of points $N^\ver_T$ as a sum over
independent sets.

\begin{corollary}
  The value at $q=1$ of the polynomial $N^\ver_T$ is the number $\operatorname{vc}(T)$
  of maximum independent sets of $T$.
\end{corollary}









\section{Cohomology: general setting and results}

This section first describes some differential forms that are always
present in the varieties under study, and then very briefly recalls the
results one needs about (mixed) Hodge structures. For a general
reference about mixed Hodge structures, see for example
\cite{mixedhodge}.

\subsection{Weil-Petersson two-form}

Let $T$ be a tree and let $S$ be a subset of $T$. Consider the
augmented tree $T+S$ obtained by adding a new edge out of every vertex
in $S$, and endow this tree with a bipartite orientation, where every
vertex is either a sink or a source.

As a variant of the definition of the variety $X_T^\phi$, one can
define a variety $X(T+S)$ attached to this data, with invertible
variables associated to the new vertices, playing the role of
coefficients in the equations (as the $\alpha$ do).

Let $\omega_i$ denote $d \log(x_i)$. The following lemma has been
proved by Greg Muller in \cite{mullerWP} in a more general context.
\begin{lemma}
  The differential form
  \begin{equation}
    \label{def_wp}
    \WP = \sum_{i \to j} \omega_i \omega_j,
  \end{equation}
  where the sum is running over edges of $T+S$, is an algebraic
  differential form on the variety $X(T+S)$.
\end{lemma}
\begin{proof}
  Let us prove that it has no pole.

  Let us fix $i$. To study the possible pole along $x_i = 0$, it is enough
  to look at the sum $\sum_{j \leftrightarrow i} \omega_i \omega_j$
  restricted to edges containing $i$.

  By the relation $x_i x'_i = 1 + \prod_{j\leftrightarrow i} x_j$, one has
  \begin{equation}
    x_i d x'_i + x'_i d x_i = \sum_{j\leftrightarrow i} \left(\prod_{{k \not= j}\atop{k\leftrightarrow i}} x_k\right) d x_j,
  \end{equation}
  and therefore
  \begin{equation}
    x_i dx'_i dx_i = \sum_{j\leftrightarrow i} \left(\prod_{{k \not= j}\atop{k\leftrightarrow i}} x_k\right) d x_j d x_i.
  \end{equation}
  This implies
  \begin{equation}
     dx'_i dx_i /  \prod_{k\leftrightarrow i} x_k = \sum_{j\leftrightarrow i}  \omega_j \omega_i,
  \end{equation}
  where the left-hand side has clearly no pole at $x_i$.
\end{proof}

Note that $\WP$ stands here for Weil-Petersson.

Abusing notations, one will use the same symbol $\WP$ to denote
these differential forms on different varieties. The ambient variety
should be clear from the context.

\subsection{Hodge structures}

We will use the notation $\QQ(-i)$ to denote a one dimensional vector
space over $\QQ$ endowed with a pure Hodge structure of Tate type, of
weight $2i$ and type $(i,i)$. The tensor product of $\QQ(-i)$ and
$\QQ(-j)$ is $\QQ(-i-j)$.

Recall that the cohomology of $\gm$ has an Hodge structure described by
\begin{equation}
  \HH^k(\gm) = \QQ(-k)
\end{equation}
for $0 \leq k \leq 1$.

There is no morphism between pure Hodge structures of distinct
weights. The Künneth isomorphism is compatible with the Hodge
structures. The Mayer-Vietoris long exact sequence is an exact
sequence of Hodge structures.

\section{Cohomology: orange and versal cases}

This section deals with the cohomology, in several cases where either
varieties do not depend on parameters, or versal conditions are
assumed on all parameters. The first part is devoted to linear trees;
the results there can then be used as building blocks.

\subsection{Linear trees $\TA$}

Let $\TA_n$ be the linear tree with $n$ vertices numbered from $1$ to
$n$. As seen in \S \ref{pointsA}, this is an orange tree if $n$ is
even, and an unimodal tree otherwise. Some of the results of this
section were already obtained in \cite{nombre_points} using instead
the cohomology with compact supports.

\subsubsection{Cohomology of some auxiliary varieties for $\TA$}

\label{auxi_section_A}

Let us introduce three varieties $X_n$, $Y_n$ and $Z_n$ with
dimensions $n, n+1$ and $n+1$.

The variety $Z_n$ is defined by variables $x_1,\dots,x_n$,
$x'_1,\dots,x'_n$ and $\alpha$ such that
\begin{align}
  x_1 x'_1 &= 1 + \alpha x_2 ,\\
  x_i x'_i &= 1 + x_{i-1} x_{i+1}, \\
  x_n x'_n &= 1 + x_{n-1}.
\end{align}

The variety $Y_n$ is the open set in $Z_n$ where $\alpha$ is invertible.

The variety $X_n$ is the closed set in $Y_n$ where $\alpha$ is fixed
to a generic invertible value (where generic means distinct from
$(-1)^{(n+1)/2}$ if $n$ is odd).

In our general notations, $Y_n$ is $X_{\TA_n}^\ver$ and $X_n$ is $X_{\TA_n}^\gen$.

Let us first describe the variety $Z_n$.

\begin{proposition}
  There exists an isomorphism between $Z_n$ and the affine space $\aff{n+1}$.
\end{proposition}
\begin{proof}
  This has been proved in \cite[Prop. 3.6]{nombre_points}.
\end{proof}

Therefore, the  cohomology of $Z_n$ is known for all $n$:
\begin{equation}
  \HH^k(Z_n) =
  \begin{cases}
    \QQ 1 &\text{ if } k=0,\\
    0 &\text{ if } k>0.
  \end{cases}
\end{equation}
The Hodge structure on $\HH^0(Z_n)$ is $\QQ(0)$.

Let us now compute the cohomology of $Y_n$ by induction. This uses the
Mayer-Vietoris long exact sequence for the covering of $Z_n$ by the
two open sets $U(x_1)$ and $U(\alpha)$.

First, let us note that $U(\alpha) \simeq Y_n $ by definition. Next,
one finds that $U(x_1) \simeq \aff{1} Y_{n-1}$. Indeed one can
eliminate $x'_1$ using the first equation. Then $\alpha$ becomes a
free variable, and there remains the equations for $Y_{n-1}$, with
$x_1$ now playing the role of $\alpha$. Last, the intersection
$U(\alpha) \cap U(x_1)$ is isomorphic to $ \gm Y_{n-1}$, by the same
argument.

Let us write $\omega_{\alpha}$ for $d \log(\alpha)$.

\begin{proposition}
  The cohomology ring of $Y_n$ has the following description:
  \begin{equation}
    \HH^k(Y_n) = \QQ(-k)
  \end{equation}
  for $0 \leq k \leq n+1$. It has a basis given by
  powers of $\WP$ in even degrees and by powers of $\WP$ times
  $\omega_{\alpha}$ in odd degrees. It is generated by the $1$-form
  $\omega_{\alpha}$ and the $2$-form $\WP$.  
\end{proposition}
\begin{proof}
  Because of the vanishing of $\HH^k(Z_n)$ for $k>0$, the
  Mayer-Vietoris long exact sequence gives short exact sequences
  \begin{equation*}
    0 \to \HH^0(Z_n) \to \HH^0(Y_n)\oplus\HH^0(U(x_1)) \to \HH^0(U(\alpha) \cap U(x_1)) \to 0,
  \end{equation*}
  and
  \begin{equation*}
    0 \to \HH^k(Y_n)\oplus\HH^k(U(x_1)) \to \HH^k(U(\alpha) \cap U(x_1)) \to 0,
  \end{equation*}
  for every $k > 0$. This determines by induction the Hodge structure
  of the cohomology of $Y_{n}$.

  Let us now proceed to the expected basis. One already knows that
  $\WP$ and $\omega_{\alpha}$ are indeed algebraic differential forms
  on $Y_{n}$.

  By the short exact sequences above, one can check that for $k>0$ the
  union of the expected basis of $\HH^k(Y_n)$ with the known basis of
  $\HH^k(U(x_1))$ is mapped to a basis of $\HH^k(U(\alpha) \cap
  U(x_1))$. This implies the statement.
\end{proof}

\subsubsection{Cohomology for $\TA_n$ with even $n$}

Let us now consider the linear tree $\TA_n$ for even $n$, and compute
the cohomology of $X_n$.
\begin{proposition}
  The Hodge structure of the cohomology of $X_n$ is
  \begin{equation}
    \HH^k(X_n) = \QQ(-k)
  \end{equation}
  for all even $k$ between $0$ and $n$, and $0$ otherwise. A basis is
  given by powers of $\WP$. The cohomology ring is generated by
  $\WP$.
\end{proposition}
\begin{proof}
  This follows from the known cohomology of $Y_n$ and the Künneth
  theorem applied to the isomorphism $Y_n \simeq X_n \gm$ given by
  lemma \ref{iso_over_gm_orange_green}. The Künneth theorem
  gives immediately the Hodge structure. 

  For the basis, it is enough to recall that the $\gm$ factor is given
  by the value of $\alpha$, and to check that fixing the value $\alpha
  = 1$ maps $\WP$ (for $Y_n$) to $\WP$ (for $X_n$).
\end{proof}

\subsection{Cohomology for orange trees of shape $H$}

\begin{center}
{\scalefont{0.7}
\begin{tikzpicture}[scale=0.7]
\usetikzlibrary{patterns,decorations.pathreplacing}
\tikzstyle{every node}=[draw,shape=circle,very thick,fill=white]

\draw (0,0) node {...} -- (1,0) node {...} -- (3,0) node {...} -- (4,0) node {...};
\draw (0,1) node {...} -- (1,1) node {...} -- (3,1) node {...} -- (4,1) node {...};

\draw (2,1) node {$a$} -- (2,0) node {$b$};

\tikzstyle{every node}=[draw=none,fill=none];
\draw [thick,decoration={
        brace,
        raise=0.3cm
    },decorate] (-0.3,1) node {} -- (1.3,1) node {};
\draw [thick,decoration={
        brace,
        raise=0.3cm
    },decorate] (2.7,1) node {} -- (4.3,1) node {};
\draw [thick,decoration={
        brace,
        mirror,
        raise=0.3cm
    },decorate] (-0.3,0) node {} -- (1.3,0) node {};
\draw [thick,decoration={
        brace,
        mirror,
        raise=0.3cm
    },decorate] (2.7,0) node {} -- (4.3,0) node {};

\node at (0.5,-1) {$m$};
\node at (0.5,2) {$k$};
\node at (3.5,-1) {$n$};
\node at (3.5,2) {$\ell$};

\end{tikzpicture}}
\end{center}

Let us denote by $H_{k,\ell,m,n}$ the tree described as two chains joined
by an edge, such that by removing the joining edge and its extremities
$a$ and $b$, one gets two chains of lengths $k$ and $\ell$ on the $a$ side
(top) and two chains of lengths $m$ and $n$ on the $b$ side (bottom).

We assume now that $H_{k,\ell,m,n}$ is an \orange tree. It implies
that either $k, \ell, m$ and $n$ are even if the middle edge is an
orange domino, or that (without loss of generality) $k$ and $m$ are
odd and $l$ and $n$ are even otherwise.

Then one can compute the cohomology of $H_{k,\ell,m,n}$ using the
Mayer-Vietoris long exact sequence for the open covering by $U(x_a)$
and $U(x_b)$. 

When the middle edge is an orange domino, one has 
\begin{equation}
  \label{iso_middle_dom}
\begin{aligned}
  U(x_a) &\simeq X_k X_\ell Y_{m+n+1},\\
  U(x_b) &\simeq Y_{k+\ell+1} X_m X_n ,\\
  U(x_a) \cap U(x_b) &\simeq (\gm)^2 X_k X_\ell X_m X_n.
\end{aligned}
\end{equation}
When the middle edge is not an orange domino, one finds instead
\begin{equation}
  \label{iso_middle_not_dom}
\begin{aligned}
  U(x_a) &\simeq Y_k X_\ell X_{m+n+1},\\
  U(x_b) &\simeq X_{k+\ell+1} Y_m X_n,\\
  U(x_a) \cap U(x_b) &\simeq Y_k X_\ell Y_m X_n.
\end{aligned}
\end{equation}
Let us introduce some notations: call $K,L,M,N$ the subsets of
vertices corresponding to the four branches of $H$ (\textit{i.e.} the
connected components of $H \setminus \{a,b\}$).

Let us denote by $W_S$ the Weil-Petersson $2$-form associated with a
subset $S$ of the vertices of $H$. For conciseness, one will use
shortcuts such as $W_{KaL}$ or $W_{MabN}$. Note that there holds
\begin{equation*}
  \omega_a W_{aL} = \omega_a W_{L}
\end{equation*}
and other similar simplifications, by the definition \eqref{def_wp} of these forms.

Let us now describe generators and bases of the cohomology of the open
sets $ U(x_a)$, $U(x_b)$ and $U(x_b)\cap U(x_b)$. This can be computed
using the isomorphisms \eqref{iso_middle_dom},
\eqref{iso_middle_not_dom} and the known cohomology of varieties $X$
and $Y$. It turns out that the result does not depend on whether or
not the middle edge $a-b$ is an orange domino.

The cohomology of $U(x_a)$ is generated by $\omega_a$, $W_{Ka}$, $W_{aL}$
and $W_{MabN}$. A basis is given by
\begin{equation}
  \label{base_ua}
  W_{Ka}^{\kappa} W_{aL}^{\lambda} W_{MabN}^{B} \quad\text{and}\quad
  \omega_a W_{K}^{\kappa} W_{L}^{\lambda} W_{MbN}^{B},
\end{equation}
where $0 \leq \kappa \leq k/2$, $0 \leq\lambda \leq l/2$ and $0 \leq B \leq (m+n+2)/2$ (left) or $0 \leq B \leq (m+n)/2$ (right). 

Similarly, the cohomology of $U(x_b)$ is generated by $\omega_b$,
$W_{Mb}$, $W_{bN}$ and $W_{KabL}$. A basis is given by
\begin{equation}
  \label{base_ub}
  W_{Mb}^{\mu} W_{bN}^{\nu} W_{KabL}^{A}  \quad\text{and}\quad
  \omega_b W_{M}^{\mu} W_{N}^{\nu} W_{KaL}^{A},
\end{equation}
where $0 \leq\mu \leq m/2$, $0 \leq\nu \leq n/2$ and $0 \leq A \leq (k+l+2)/2$ (left) or $0 \leq A \leq (k+l)/2$ (right).

The cohomology of $U(x_b)\cap U(x_b)$ is generated by $\omega_a$,
$\omega_b$, $W_{Mb}$, $W_{bN}$, $W_{Ka}$ and $W_{aL}$. A basis is given by
\begin{equation}
  \label{base_ua_ub}
\begin{aligned}
  W_{Ka}^{\kappa} W_{aL}^{\lambda} W_{Mb}^{\mu} W_{bN}^{\nu},\quad
  \omega_a W_{K}^{\kappa} W_{L}^{\lambda} W_{Mb}^{\mu} W_{bN}^{\nu},\\
  \omega_a \omega_b W_{K}^{\kappa} W_{L}^{\lambda} W_{M}^{\mu} W_{N}^{\nu}
  \quad\text{and}\quad
  \omega_b W_{Ka}^{\kappa} W_{aL}^{\lambda} W_{M}^{\mu} W_{N}^{\nu},
\end{aligned}
\end{equation}
with the same conditions as above on $\kappa,\lambda,\mu$ and $\nu$.

There is a bigrading corresponding to the top and bottom parts of
the $H$ shape. Every differential form involved in the bases just
described is a sum of products of $\omega_i$. The bidegree of a
monomial in the $\omega_i$ is the pair (number of $\omega_i$ where $i$
is in the top row, number of $\omega_i$ where $i$ is in the
bottom row). Among the various Weil-Petersson forms involved,
only the differential forms $W_{KabL}$ and $W_{MabN}$ are not
homogeneous for the bidegree, but have terms in bidegrees $(2,0)$ and
$(1,1)$ (resp. $(0,2)$ and $(1,1)$).

One needs now to compute explicitly the following maps in the
Mayer-Vietoris long exact sequence:
\begin{equation*}
  \HH^i(U(x_a))\oplus\HH^i(U(x_b)) \stackrel{f_i}{\longrightarrow} \HH^i(U(x_a) \cap U(x_b)).
\end{equation*}
Because one has bases of all these spaces, this is a matter of matrices.



\smallskip

For odd degree $i$, let us show that the differential is
injective. Because in this case all basis elements (given by right
columns of \eqref{base_ua}, \eqref{base_ub} and \eqref{base_ua_ub})
are homogeneous for the bigrading, one can separate the cases of
bidegree congruent to $(0,1)$ and to $(1,0)$ modulo $(2,2)$. Let us
give details only for the first possibility, the other case being
similar after exchanging top and bottom of $H$. The basis of the
corresponding bihomogeneous subspace of $\HH^i(U(x_b))$ is given by
$\omega_{b}W_{KaL}^{A}W_{M}^{\mu} W_{N}^{\nu}$ with $i =
1+2A+2\mu+2\nu$. The corresponding bihomogeneous subspace of
$\HH^i(U(x_a))$ is zero. The basis of the corresponding bihomogeneous
subspace of $\HH^i(U(x_a) \cap U(x_b))$ is given by
$\omega_{b}W_{Ka}^{\kappa}W_{aL}^{\lambda} W_{M}^{\mu} W_{N}^{\nu}$
with $i = 1+2\kappa+2\lambda+2\mu+2\nu$. But $W_{KaL}^{A}$ can be
written as a linear combination of $W_{Ka}^{\kappa}W_{aL}^{\lambda}$
with $\kappa + \lambda = A$. Therefore the basis elements are mapped
to linear combinations with disjoint supports. It follows that the map
$f_i$ is injective.

\smallskip

Let us now turn to even degrees.
\begin{proposition}
  For even degree $2i$, the kernel of the differential $f_{2i}$ has
  dimension $1$, spanned by the $i^{th}$ power of the form $\WP$.
\end{proposition}
\begin{proof}
  First note that one can define an injective map $\Delta$ from the
  space $\HH^{2i}(U(x_a) \cap U(x_b))$ to the space $D_i$ spanned by
  all products of $i$ $2$-forms of the shape $\omega_{s}\omega_{t}$
  for $s-t$ an edge of the tree (always written in the order given by
  a fixed alternating orientation of the tree). Indeed, both terms in
  the left column of \eqref{base_ua_ub} can be written as linear
  combinations of such products. The injectivity holds because
  distinct elements in this part of the basis are mapped to linear
  combinations with disjoint supports. To recover a basis element $B$
  from any monomial in its image by $\Delta$, first count in
  $\Delta(B)$ if the number of $\omega_k$ in the top row is odd or
  even. This tells if the basis elements $B$ contains
  $\omega_a\omega_b$ or not. Then it is easy to recover the exponents
  $(\kappa, \lambda, \mu, \nu)$ defining $B$ by counting in
  $\Delta(B)$ how many $\omega_k$ there are in the different parts of
  the tree.

  To prove the statement of the proposition, it is therefore enough to
  compute the kernel of the composite map $\Delta \circ f_{2i}$.

  It turns out that the matrix of this composite map has a nice
  description. First, every monomial $d$ made of $i$ $2$-forms
  $\omega_{s}\omega_{t}$ as above appears in exactly two images, the
  image of a form $W_{Ka}^{\kappa} W_{aL}^{\lambda} W_{MabN}^{B}$ and
  the image of a form $W_{Mb}^{\mu} W_{bN}^{\nu} W_{KabL}^{A}$ (with
  opposite signs). Let us denote these two forms by $\mathsf{F}_a(d)$ and
  $\mathsf{F}_b(d)$. On the other hand, the image of every basis element is
  the sum of several monomials (at least one), with constant sign.

  Let us pick an element $z$ of the kernel of $f_{2i}$. Then for every
  monomial $d$ in $D_i$, the coefficients of $\mathsf{F}_a(d)$ and
  $\mathsf{F}_b(d)$ in $z$ must be the same. One can make a graph with
  vertices given by all forms in the basis, and edges corresponding to
  the relations $\mathsf{F}_a(d)-\mathsf{F}_b(d)$ for all monomials
  $d$.

  By a combinatorial argument, one can check that this graph is
  connected. For this, one just has to show that one can go from any
  monomial $d$ to any monomial $d'$, using two kinds of moves: replace
  $d$ by another monomial appearing in the same $\mathsf{F}_a(d)$, or
  replace $d$ by another monomial appearing in the same
  $\mathsf{F}_b(d)$. This is not difficult once translated in terms of
  dominoes, and details are left to the reader.

  From the connectedness of this graph, one deduces that the kernel is
  spanned by the sum of all basis elements of $\HH^{2i}(U(x_a))\oplus
  \HH^{2i}(U(x_b))$, which is just $(\WP^i,\WP^i)$.




\end{proof}

This proposition and the injectivity in the case of odd
degree allow to give a description of the weights of the Hodge
structure on the cohomology. This can easily be made explicit, but one
will not do that here.

There would remain to find explicit expressions for the cohomology
classes coming from the co-image of the differentials $f_i$.

\medskip

In the case of the Dynkin diagrams $\TE_6$ and $\TE_8$, one can go
further and compute explicit representatives of the cohomology classes.

By the general proof, the cohomology for $\TE_6$ is described by
\begin{equation*}
  \QQ(0) \mid 0 \mid \QQ(-2) \mid 0 \mid \QQ(-3)\oplus \QQ(-4) \mid 0 \mid \QQ(-6),
\end{equation*}
where the $\QQ(-i)$ with $i$ even correspond to the powers of $\WP$.

Using the connection homomorphism in the long exact sequence, one finds that
the form
\begin{equation}
  dx_2 dx_3 dx_5 \omega_4
\end{equation}
corresponds to $\QQ(-3)$.





Similarly, the cohomology for $\TE_8$ is described by
\begin{equation*}
  \QQ(0) \mid 0 \mid \QQ(-2) \mid 0 \mid \QQ(-3)\oplus \QQ(-4) \mid 0 \mid \QQ(-5)\oplus \QQ(-6) \mid 0 \mid \QQ(-8),
\end{equation*}
where the even $\QQ(-i)$ are the powers of $\WP$.

One finds that the form
\begin{equation}
  dx_2 dx_3 dx_5 \omega_4  
\end{equation}
corresponds to $\QQ(-3)$, and its product by $\WP$ corresponds to $\QQ(-5)$.

\begin{center}
{\scalefont{0.7}
\begin{tikzpicture}[scale=0.7]
\tikzstyle{every node}=[draw,shape=circle,very thick,fill=white]

\draw (0,0) node[fill=orange!20] {1} -- (1,0) node[fill=orange!20] {2} -- (2,0) -- (3,0) node[fill=orange!20] {5} -- (4,0) node[fill=orange!20] {6} -- (5,0) node[fill=orange!20] {7} -- (6,0) node[fill=orange!20] {8};

\draw (2,1) node[fill=orange!20] {3} -- (2,0) node[fill=orange!20] {4};

\end{tikzpicture}}
\end{center}

\section{Cohomology: generic cases}

This section contains one conjecture and one result in some specific
cases about the cohomology of generic fibers.

\subsection{Cohomology for $\TA$ odd and generic}

Let us now consider the linear tree $\TA_n$ for odd $n$, which is
unimodal. In this section, one proposes a conjectural description for
the cohomology of the variety $X_{\TA_n}^\gen$ (which is also denoted
$X_n$ in \S \ref{auxi_section_A}).

\begin{conjecture}
  The Hodge structure on the cohomology of $X_n$ is given by
  \begin{equation}
    \HH^k(X_n) = \QQ(-k)
  \end{equation}
  for even $k$ in $0 \leq k \leq (n-1)$, and
  \begin{equation}
    \HH^{n}(X_n) = \oplus_{i=(n+1)/2}^{n} \QQ(-i).
  \end{equation}
  The cohomology ring has a basis given by all powers $\WP^i$ for
  $0 \leq i \leq (n-1)/2$ and by a basis of $\HH^{n}(X_n)$. The
  cohomology ring is generated by $\WP$ in degree $2$ and by the
  elements of $\HH^{n}(X_n)$ in degree $n$.
\end{conjecture}

One approach for this computation would be using the covering of $X_n$
by the $(n+1)/2$ open sets $U(x_i)$ ($i$ odd) given by Lemma
\ref{admi_cover}. One can then consider the spectral sequence for this
covering (where $d_1$ is the deRham differential and $d_2$ is the Cech
differential).

The intersection of open sets in this covering have a simple
description: they are products $\gm$ times two varieties of the type
$X_k$ with $k$ even, times some varieties of type $Y_k$ with $k$ odd.


\begin{lemma}
  This spectral sequence degenerates at $E_2$.
\end{lemma}
\begin{proof}
  This follows from the purity of the Hodge structure on the
  cohomology of the open sets in the covering.
\end{proof}

It would therefore be enough to understand the behavior of the Cech
differential acting on the cohomology groups of the open sets. This is
still a rather intricate question. The conjecture has been checked by
computer for $n \leq 11$. Maybe one should look for a better approach.


\begin{remark}
  To give an explicit description of the generators of the top
  cohomology group seems to be an interesting problem.
\end{remark}

\subsection{Cohomology for $\TD$ odd and generic}

Let us now consider the tree $\TD_n$ for odd $n$, which is
unimodal. Our aim is to compute the cohomology of the variety
$X_{\TD_n}^\gen$.

One will assume that the generic parameter $\alpha$ is attached to the
vertex $1$, where $1$ and $2$ are the two red vertices on the short
branches. By Lemma \ref{admi_cover}, one has a covering by $U(x_1)$
and $U(x_2)$. One will use the Mayer-Vietoris long exact sequence for
this covering. One has
\begin{align*}
  U(x_1) &\simeq \gm X_{n-1},\\
  U(x_2) &\simeq \gm X_{n-1},\\
  U(x_1) \cap U(x_2) & \simeq \gm Y_{n-2}.
\end{align*}

Given the known explicit description of the cohomology rings of
$X_{n-1}$ and $Y_{n-2}$, one can write very explicitly the long exact
sequence.

First note that the Hodge structure of $\HH^k(U(x_1))\oplus
\HH^k(U(x_2))$ is $2\,\QQ(-k)$ for $0\leq k \leq n$. Similarly, the Hodge
structure of $\HH^k(U(x_1) \cap U(x_2))$ is $2\,\QQ(-k)$, unless $k=0$ or
$n$ where it is $\QQ(-k)$.

Using the known basis of the cohomology, one can describe the map
$\rho_k$ from $\HH^k(U(x_1))\oplus \HH^k(U(x_2))$ to $\HH^k(U(x_1)
\cap U(x_2))$. One can see that this map has rank $1$ if $k$ is
even. One can also check that it is an isomorphism if $k$ is odd,
unless $k=n$ where it has rank $1$.

It follows that the Hodge structure on $\HH^k(X_{\TD_n}^\gen)$ is given by
\begin{equation}
  \begin{cases}
    \QQ(-k) \quad& \text{if} \quad k\equiv 0\, (\mod 2),\\
    \QQ(-k+1) \quad &\text{if} \quad k \equiv 1\, (\mod 2),\, k \not\in \{1,n\}\\
    \QQ(-n+1)\oplus \QQ(-n) \quad &\text{if}  \quad k=n.
  \end{cases}
\end{equation}

Moreover, it also follows from the explicit knowledge of the long
exact sequence that the classes in even cohomological degree are just
the powers of the $2$-form $\WP$.

One can also see that the Hodge structure $\QQ(-n)$ in cohomological
degree $n$ is given by the differential form $\Lambda_{i=1}^{n} \omega_i$.

There remains to understand the even Hodge structures present in odd
cohomological degrees.

By a small diagram chase, and using the formula
\begin{equation}
  \frac{1-\alpha}{x_1 x_2} = \frac{x'_1}{x_2} - \alpha \frac{x'_2}{x_1},
\end{equation}
one finds that a basis of the $\QQ(-2)$ part of $\HH^3(X_{\TD_n}^\gen)$ is
given by the differential form
\begin{equation}
  dx_3 \omega_1 \omega_2.
\end{equation}
Moreover, a similar computation shows that products of this form by
powers of $\WP$ give a basis for the even Hodge structures in odd
cohomological degrees.

The cohomology ring is therefore generated by one generator in each
degree $2$, $3$ and $n$ (of Hodge type $\QQ(-2)$, $\QQ(-2)$ and $\QQ(-n)$).

\appendix
\section{Algorithm for the canonical coloring of trees}

\label{algosection}

Let us now describe an algorithm to find the \red-\orange-\green
coloring. Let $T$ be a tree.

\begin{enumerate}
\item At start, all vertices are considered to be \red.

\item Then, one changes the colors according to the following rule:

  If a vertex $v$ has exactly one \red neighbor $w$, this \red neighbor
  becomes \green.

  If moreover $v$ is \green, then one puts a domino on the edge $v-w$.

\item One repeats the previous step until no color can change.

\item Then one colors in \orange the \green vertices that do not have
  a \red neighbor.
\end{enumerate}

One gets in that way a coloring of the tree with \green, \orange and
\red vertices, together with a collection of dominoes.

\begin{proposition}
  This algorithm defines the same coloring as in section \ref{section1}. Moreover the dominoes
  obtained are those that are present in all maximum matchings.
\end{proposition}
\begin{proof}
  At the end of step $3$, one has obtained a tree with red and green
  vertices, with the property that every vertex has either no red
  neighbor or at least two red neighbors.

  Let us prove that a red vertex can not have at least two red
  neighbors. Assume that there is such a vertex $v_1$. Let $v_2$ be
  one of its red neighbors. Then $v_2$ must also have at least two red
  neighbors. Hence one can find another red neighbor $v_3$ of
  $v_2$. Going on in this way, and because $T$ is a tree, one can
  build an infinite sequence of red vertices, which is absurd.

  So, after step $3$, one has three kinds of vertices: red vertices
  (they have only green neighbors), green vertices with no red
  neighbors and green vertices with at least two red neighbors.
  
  It follows that after step $4$, one has the following situation: red
  vertices with only green neighbors, green vertices with at least two
  red neighbors, and orange vertices with no red neighbors.

  Using the third characterization of the coloring, it just remains to
  prove that the induced forest on orange vertices has a perfect
  matching. This matching is provided by the set of dominoes computed
  by the algorithm. When a domino is introduced, both its vertices are
  green. We need a lemma.

  \begin{lemma}
    During the algorithm, the configuration
    \begin{equation*}
      \colorbox{red!20}{$u$} - \colorbox{green!20}{$v$} - \colorbox{green!20}{$w$}
    \end{equation*}
    where $u$ is red and $v-w$ is a domino, does not appear.
  \end{lemma}
  \begin{proof}
    Let us assume the contrary, and let $u-v-w$ be such a configuration.

    Because $v$ still has a red neighbor, the domino $v-w$ must have
    been created by turning green the vertex $v$ as the last red
    neighbor of the green vertex $w$.

    Let us go back to this previous step of this algorithm, where $u$
    and $v$ are red, $w$ is green with $v$ as only red neighbor.
    \begin{equation*}
      \colorbox{red!20}{$u$} - \colorbox{red!20}{$v$} - \colorbox{green!20}{$w$}
    \end{equation*}

    So $w$ must have another neighbor $z$, such that $w$ has turned
    green as the last red neighbor of $z$.
    \begin{equation*}
      \colorbox{red!20}{$u$} - \colorbox{red!20}{$v$} - \colorbox{green!20}{$w$}- \colorbox{green!20}{$z$}
    \end{equation*}

    One can assume, by changing maybe the order in which the algorithm
    has been performed, that $z$ has turned green before $w$. This is
    because trees are bipartite, and the algorithm can be run
    independently on the two parts of the bipartition.

    Therefore, $w$ has turned green as the last red neighbor of
    the green vertex $z$, and hence belongs to a domino $w-z$. Hence
    one has found a configuration $v-w-z$ similar to the initial one:
    \begin{equation*}
      \colorbox{red!20}{$v$} - \colorbox{green!20}{$w$} - \colorbox{green!20}{$z$}.
    \end{equation*}

    This can be iterated to provide an infinite sequence of
    vertices. This is absurd.
  \end{proof}
  
  It follows from the lemma that once a domino is created, its
  vertices do not have any red neighbors. Therefore they will be
  orange at the end. 

  This also implies that the dominoes are disjoint, because the
  creation of a domino takes a red vertex with only green neighbors
  and a green vertex with exactly one red neighbor, and produces a
  pair of green vertices with only green neighbors. Therefore a vertex can
  only enter once in a domino.

  Moreover, every orange vertex $v$ is in a domino. This is because
  green vertices surrounded only by green vertices can only be
  introduced during the creation of a domino.  
\end{proof}

\begin{remark}
  From the previous proof, one can see that one can modify the
  algorithm as follows: when creating a new domino, color in \orange
  its two vertices, and forget step $4$.
\end{remark}

\bibliographystyle{alpha}
\bibliography{cohomolo}

\end{document}